\documentclass[11pt]{article}
\usepackage{amsfonts}
\usepackage{amsthm}
\input{epsf.sty}
\usepackage{latexsym,amsmath,amsfonts,amscd}
\usepackage{epsfig}
\usepackage{changebar}
\usepackage{pstricks}
\usepackage{pst-plot}
\usepackage{multirow}
\usepackage{subfigure}
\usepackage{placeins}
\usepackage{amssymb}
\usepackage{mathrsfs}
\usepackage[title]{appendix}
\usepackage{bm}
\usepackage{hyperref}
\usepackage{geometry}
\usepackage{multirow,bigstrut}
\usepackage{enumitem}
\usepackage{booktabs}
\usepackage{enumitem}

\numberwithin{equation}{section}

\theoremstyle{plain}
\newtheorem{thm}{Theorem}[section]
\newtheorem{lem}[thm]{Lemma}

\newtheorem{defn}[thm]{Definition}
\newtheorem{rem}[thm]{Remark}

\theoremstyle{definition}
\newtheorem{exam}[thm]{Example}

\newcommand{\brac}[1]{\left(#1\right)}
\newcommand{\abs}[1]{\left\vert#1\right\vert}

\newcommand{\vect}[1]{\boldsymbol{#1}}

\allowdisplaybreaks[4]
\begin{document}
\baselineskip=1.6pc

\vspace{.5in}

\begin{center}

{\large\bf Machine learning moment closure models for the radiative transfer equation III: enforcing  hyperbolicity and physical characteristic speeds}

\end{center}

\vspace{.1in}

\centerline{
Juntao Huang \footnote{Department of Mathematics,
Michigan State University, East Lansing, MI 48824, USA.
E-mail: huangj75@msu.edu} \qquad
Yingda Cheng  \footnote{Department of Mathematics, Department of Computational Mathematics, Science and Engineering, Michigan State University, East Lansing, MI 48824, USA. E-mail: ycheng@msu.edu. Research is supported by NSF grants    DMS-2011838 and AST-2008004.} \qquad
Andrew J. Christlieb \footnote{Department of Computational Mathematics, Science and Engineering, Michigan State University, East Lansing, Michigan 48824, USA. E-mail: christli@msu.edu. Research is supported by: AFOSR grants FA9550-19-1-0281 and FA9550-17-1-0394; NSF grants DMS-1912183 and AST-2008004; and DoE grant DE-SC0017955.} \qquad
Luke F. Roberts \footnote{National Superconducting Cyclotron Laboratory and Department of Physics and Astronomy, Michigan State University, East Lansing, MI 48824, USA and CCS-2, Los Alamos National Laboratory, Los Alamos, NM 87545, USA. E-mail: lfroberts@lanl.gov. Research is supported by: NSF grant AST-2008004; and DoE grant DE-SC0017955.}
}

\vspace{.4in}

\centerline{\bf Abstract}
\vspace{.1in}

This is the third paper in a series in which we develop machine learning (ML) moment closure models for the radiative transfer equation (RTE). 
In our previous work \cite{huang2021gradient}, we proposed an approach to learn the gradient of the unclosed high order moment, which performs much better than learning the moment itself and the conventional $P_N$ closure.
However, while the ML moment closure has better accuracy, it is not able to guarantee hyperbolicity and has issues with long time stability. In our second paper \cite{huang2021hyperbolic}, we identified a symmetrizer which leads to conditions that enforce that the gradient based ML closure is symmetrizable hyperbolic and stable over long time. The limitation of this approach is that in practice the highest moment can only be related to four, or fewer, lower moments.  

In this paper, we propose a new method to enforce the hyperbolicity of the ML closure model. Motivated by the observation that the coefficient matrix of the closure system is a lower Hessenberg matrix, we relate its eigenvalues to the roots of an associated polynomial. 
We design two new neural network architectures based on this relation. 
The ML closure model resulting from the  first neural network is weakly hyperbolic and guarantees the physical characteristic speeds, i.e., the eigenvalues are bounded by the speed of light. The second model is strictly hyperbolic and does not guarantee the boundedness of the eigenvalues.
Several benchmark tests including the Gaussian source problem and the two-material problem show the good accuracy, stability and generalizability of our hyperbolic ML closure model.

\bigskip


\pagenumbering{arabic}

\section{Introduction}\label{sec:intro}
\setcounter{equation}{0}
\setcounter{figure}{0}
\setcounter{table}{0}
In this paper, we introduce an extension to our previous works on ML closures for  radiative transfer  modeling \cite{huang2021gradient,huang2021hyperbolic}. The new approach enforces the hyperbolicity (and physical characteristic speeds) by ensuring mathematical consistency between the closure and the macroscopic model. Further, the numerical results  demonstrate the plausibility of capturing kinetic effects in a moment system with a handful of moments and an appropriate closure model.  

The study of radiative transfer is of vital importance in many fields of science and engineering including astrophysics \cite{pomraning2005equations}, heat transfer \cite{koch2004evaluation}, and optical imaging \cite{klose2002optical}. The kinetic description of  radiative transfer is a integro-differential equation in six dimensions in spatial and angular spaces plus time.  While there exist many numerical methods to solve this equation, from Monte Carlo methods to deterministic mesh based schemes, the fundamental fact remains that radiative transfer equation (RTE) is computationally demanding for many problems.  

An alternative approach is to directly model the observables of the kinetic equations: density, momentum, energy etc. by taking moments of the kinetic equation.  However, the resulting system of equations is not closed,  since the equation for the $p^{th}$ moment depends on knowledge of the $(p+1)^{th}$ moment.  This is known as the moment closure problem.  To obtain a closed system of equations, typically a relationship has to be introduced to eliminate  the dependency of the equations on the $(p+1)^{th}$ moment. This may be as simple as setting the  $(p+1)^{th}$ moment  to zero, or involve some other relations relating the $(p+1)^{th}$ moment to lower order moments.   Many moment closure models have been developed, including the $P_N$ model \cite{chandrasekhar1944radiative}; the variable Eddington factor models \cite{levermore1984relating,murchikova2017analytic}; the entropy-based $M_N$ models \cite{hauck2011high,alldredge2012high,alldredge2014adaptive};  the positive $P_N$ models \cite{hauck2010positive}; the filtered $P_N$ models \cite{mcclarren2010robust,laboure2016implicit}; the $B_2$ models \cite{alldredge2016approximating}; and the $MP_N$ model \cite{fan2020nonlinear,fan2020nonlinear2,li2021direct}. 

In moment closure problems, hyperbolicity is a critical issue, which is essential for a system of first-order partial differential equations (PDEs) to be well-posed \cite{serre1999systems}. The pioneering work on the moment closure for the Boltzmann equation, in the context of gas kinetic theory, was introduced by Grad in \cite{grad1949kinetic} and is the most basic one among the moment models.  Recent analysis for  Grad's 13-moment model showed that the equilibrium of the model is on the boundary of the region of hyperbolicity in 3D  \cite{cai2014hyperbolicity}. This instability issue has led to a range of efforts to develop closures that lead to globally hyperbolic moment systems \cite{cai2013globally,cai2014globally,fan2020nonlinear,fan2020nonlinear2,li2021direct}.

The traditional trade off in introducing a closure relation and solving a moment model instead of a kinetic equation is generic accuracy verses practical computability. However, thanks to the rapid development of machine learning (ML)  and data-driven modeling \cite{brunton2016discovering,raissi2019physics,han2018solving}, a new approach to solve the moment closure problem has emerged based on ML \cite{han2019uniformly,scoggins2021machine,huang2020learning,bois2020neural,ma2020machine,wang2020deep,maulik2020neural,huang2021gradient,huang2021hyperbolic,hauck2021entropy,hauck2021boltzmann}. This approach offers a path for multi-scale problems that is relatively unique, promising to capture kinetic effects in a  moment model with only a handful of moments.  For more detailed literature review, we refer readers to \cite{huang2021gradient}. We   remark that most of the works mentioned above are not able to guarantee hyperbolicity or long time stability, except the works in \cite{huang2020learning,huang2021hyperbolic,hauck2021entropy,hauck2021boltzmann}. 
In \cite{huang2020learning}, based on the conservation-dissipation formalism \cite{zhu2015conservation} of irreversible thermodynamics, the authors proposed a stable ML closure model with hyperbolicity and Galilean invariance for the Boltzmann BGK equation. Nevertheless, the model is limited to only one extra non-equilibrium variable and it is still not clear how to generalize to an arbitrary number of moments. In \cite{hauck2021entropy,hauck2021boltzmann}, the authors constructed ML surrogate models for the maximum entropy closure \cite{levermore1996moment} of the moment system of the RTE and the Boltzmann equation. By approximating the entropy using convex splines and input convex neural networks \cite{amos2017input}, the ML model preserves the structural properties of the original system and reduces the computational cost of the associated ill-conditioned constrained optimization problem significantly, which needed to be solved at each time step in the original formulation of the maximum entropy closure.

This paper is a continuation of our previous work in \cite{huang2021gradient}, where we proposed to directly learn a closure that relates the gradient of the highest order moment to the gradients of the lower order moments. This gradient based closure is consistent with the exact closure for the free streaming limit and also provides a natural output normalization. A variety of numerical tests show that the ML closure model in \cite{huang2021gradient} has better accuracy than an ML closure based on learning a relation between the moments, as opposed to a relation between the gradients, and the conventional $P_N$ closure. Further, the method was able to accurately model both the optically thin and optically thick regime in a single domain with only six moments and was in good agreement with moments computed from the kinetic solution.  However, it is not able to guarantee hyperbolicity and long time simulations are not always satisfactory.

In our follow-up work \cite{huang2021hyperbolic}, we proposed a method to enforce the global hyperbolicity of the ML closure model. The main idea is to seek a symmetrizer (a symmetric positive definite matrix) for the closure system, and derive constraints such that the system is globally symmetrizable hyperbolic. It was also shown that the hyperbolic ML closure system inherits the dissipativeness of the RTE and preserves the correct diffusion limit as the Knunsden number goes to zero. In the numerical tests, the method preformed as well as our original gradient based ML closure for short time simulations and also has the additional benefit of long time stability. A limitation of our  approach in  \cite{huang2021hyperbolic} is that in practice it is limited to relating the gradient of the highest moment to the gradient of the next 4 lower  moments. However,  our analysis in  \cite{huang2021gradient} indicated that in the free streaming limit, the gradient of the highest moment should be related to  a range of gradients which include the lowest moments. 

In this paper, to overcome this limitation, we take a different approach to enforce the hyperbolicity of our gradient based ML closure model.  The approach is to design a structure preserving neural network that ensures that the desired hyperbolicity is preserved in our ML gradient based closure.    The main idea is motivated by the observation that the coefficient matrix of the gradient based closure system \cite{huang2021gradient} is an unreduced lower Hessenberg matrix, see Definition \ref{defn:hessenberg}. Due to this particular mathematical structure, we relate its eigenvalues to the roots of some polynomials associated with the coefficient matrix. Therefore, the hyperbolicity of the closure model is equivalent to the condition that the associated polynomial only has simple and real roots, see Theorem \ref{thm:real-diagonalizable} and Theorem \ref{thm:ml-real-diagonalizable}. Then, we derive the relation between the eigenvalues and the weights in the gradient based closure using the Vieta's formula and a linear transformation between monomial basis functions and Legendre polynomials.
Based on this relation, we design two new neural network architectures both starting with a fully connected neural network which takes the input as the lower order moments.
The first neural network architechture is then followed by a component-wise hyperbolic tangent function to enforce the boundedness of the eigenvalues, while the second one has some postprocessing layers to enforce that the eigenvalues are distinct. 
Lastly, two sublayers representing the Vieta's formula and a linear transformation are applied to produce the weights in the gradient based closure as the final output, see Figure \ref{fig:schematic-nn} and Figure \ref{fig:schematic-nn-distinct} in Section \ref{sec:nn}. 
The resulting ML closure model from the  first neural network is weakly hyperbolic and guarantees the physical characteristic speeds, i.e. the eigenvalues lie in the range of the interval $[-1, 1]$, see Theorem \ref{thm:nn-bound}, while the symmetrizer approach in \cite{huang2021hyperbolic} usually violates the physical characteristic speeds. 
The second model is strictly hyperbolic and does not guarantee the boundedness of the eigenvalues, see Theorem \ref{thm:nn-distinct}. Nevertheless, in practice, we find the characteristic speeds stay close to the physical bound.
Maintaining physical characteristic speeds saves substantial computational efforts by allowing for a larger time step size, as compared to \cite{huang2021hyperbolic} when solving the closure system. 
We numerically tested that the hyperbolic ML closure model has good accuracy in a variety of numerical examples and, just as with our previous work, can capture accurate solutions to problems which have regions in both the optically thin and optically thick regime with only 6 moments.  Further, we numerically demonstrate that as we increase the number of moments in the new approach, the ML closure converges rapidly to the solution of the kinetic equation.

Nevertheless, there exists some numerical instability for the current model when a small number of moments are used.
For the first neural network, we observe numerically that the eigenvalues get too close, which behaves as if the system is weakly hyperbolic instead of strongly hyperbolic. For the second neural network, we check the linear stability of the system numerically and find that the loss of linear stability probably results in the blow up of the numerical solutions. How to stabilize the closure system, while maintaining the accuracy, is a topic to be investigated in the future.

The remainder of this paper is organized as follows. In Section, \ref{sec:hessenberg}, we present some preliminary results about Hessenberg matrixes. In Section \ref{sec:moment-closure}, we introduce the hyperbolic ML moment closure model. In Section \ref{sec:nn}, we present the details in the architectures and the training of the neural networks. The effectiveness of our ML closure model is demonstrated through extensive numerical results in Section \ref{sec:numerical-test}. Some concluding remarks are given in Section \ref{sec:conclusion}.

\section{Preliminary results about Hessenberg matrix}\label{sec:hessenberg}

In this section, we review important properties of the Hessenberg matrix.  These properties facilitate directly relating the eigenvalues of a Hessenberg matrix to the roots of some associated polynomial and derive some equivalent conditions for a Hessenberg matrix to be real diagonalizable.  As the matrix being real diagonalizable is equivalent to enforcing that the first-order system is hyperbolic, this is a critical aspect  in the design of our structure-preserving neural network in Sections 3 and 4.

We start with the definitions of the (unreduced) lower Hessenberg matrix and the associated polynomial sequence \cite{elouafi2009recursion}:
\begin{defn}[lower Hessenberg matrix]\label{defn:hessenberg}
	The matrix $H = (h_{ij})_{n\times n}$ is called lower Hessenberg matrix if $h_{ij}=0$ for $j>i+1$. It is called unreduced lower Hessenberg matrix if further $h_{i,i+1}\ne0$ for $i=1,2,\cdots,n-1$.
\end{defn}
\begin{defn}[associated polynomial sequence \cite{elouafi2009recursion}]\label{defn:associated-polynomial}
	Let $H = (h_{ij})_{n\times n}$ be an unreduced lower Hessenberg matrix. The associated polynomial sequence $\{q_i\}_{0\le i\le n}$ with $H$ is defined as: $q_0=1$, and
	\begin{equation}\label{eq:recurrence-polynomial}
		q_i(x) = \frac{1}{h_{i,i+1}} \brac{x q_{i-1}(x) - \sum_{j=1}^i h_{ij} q_{j-1}(x)}, \quad 1\le i\le n,
	\end{equation}
with $h_{n,n+1}:=1$.	
\end{defn}
Notice that the recurrence relation in \eqref{eq:recurrence-polynomial} can be written as a matrix-vector form:
\begin{equation}\label{eq:recurrence-polynomial-matx-vect}
	H \vect{q}_{n-1}(x) = x \vect{q}_{n-1}(x) - q_n(x) \vect{e}_n,
\end{equation}
where $\vect{q}_{n-1}(x) = (q_0(x), q_1(x), \cdots, q_{n-1}(x))^T$ and $\vect{e}_n=(0, 0, \cdots, 0, 1)^T\in\mathbb{R}^n$. From this relation, one can immediately relate the roots of $q_n$ to the eigenvalues of $H$ \cite{elouafi2009recursion}:
\begin{thm}[\cite{elouafi2009recursion}]\label{thm:root-is-eigenvalue}
	Let $H = (h_{ij})_{n\times n}$ be an unreduced lower Hessenberg matrix and $\{q_i\}_{0\le i\le n}$ is the associated polynomial sequence with $H$. The following conclusion holds true:
	\begin{enumerate}
		\item 
		If $\lambda$ is a root of $q_n$, then $\lambda$ is an eigenvalue of the matrix $H$ and a corresponding eigenvector is $(q_0(\lambda), q_1(\lambda), \cdots, q_{n-1}(\lambda))^T$;
		\item 
		If all the roots of $q_n$ are simple, then the characteristic polynomial of $H$ is precisely $\rho q_n$ with $\rho=\Pi_{i=1}^{n-1}h_{i,i+1}$, i.e.,
		\begin{equation}
			\det(x I_n - H) = \rho q_n(x),
		\end{equation}
		where $I_n$ denotes the identity matrix of order $n$.
	\end{enumerate}	
\end{thm}

By analyzing the eigenspace of the unreduced lower Hessenberg matrix, we have the following equivalent conditions for an unreduced lower Hessenberg matrix to be real diagonalizable. The proof is included in the appendix.
\begin{thm}\label{thm:real-diagonalizable}
	Let $H = (h_{ij})_{n\times n}$ be an unreduced lower Hessenberg matrix and $\{q_i\}_{0\le i\le n}$ is the associated polynomial sequence with $H$. The following conditions are equivalent:
	\begin{enumerate}
		\item $H$ is real diagonalizable;
		\item all the eigenvalues of $H$ are distinct and real;
		\item all the roots of $q_n$ are simple and real.
	\end{enumerate}	
\end{thm}

\section{Moment closure for radiative transfer equation}\label{sec:moment-closure}

In this section, we first review the gradient based ML moment closure method for the RTE in slab geometry proposed in \cite{huang2021gradient}. Then, we present our approach to enforce the hyperbolicity of the ML moment closure model.  Our method for enforcing hyperbolicity comes from a direct relation we derive in Section \ref{sec:hmoment-closure} between the coefficients of the neural network in the gradient based model and the eigenvalues of coefficient matrix.  Given this relation, in Section \ref{sec:nn} we propose two neural network architectures where we directly learn the eigenvalues of the coefficient matrix $A$ such that the eigenvalues are real. The resulting setup produces distinct  eigenvalues  and there by guarantees that the learned gradient based closure is hyperbolic.

\subsection{Gradient based ML moment closure}

We consider the time-dependent RTE for a gray medium in slab geometry:
\begin{equation}\label{eq:rte}
	\partial_t f + v \partial_x f = {\sigma_s}\brac{\frac{1}{2} \int_{-1}^1 f dv - f} - \sigma_a f, \quad -1\le v\le 1
\end{equation}
Here, $f=f(x,v,t)$ is the specific intensity of radiation. The variable $v\in[-1, 1]$ is the cosine of the angle between the photon velocity and the $x$-axis. $\sigma_s = \sigma_s(x)\ge 0$ and $\sigma_a = \sigma_a(x)\ge 0$ are the scattering and absorption coefficients.

Denote the $k$-th order Legendre polynomial by $P_k = P_k(x)$. Define the $k$-th order moment by
\begin{equation}
	m_k(x,t) = \frac{1}{2} \int_{-1}^1 f(x,v,t) P_k(v) dv.
\end{equation}
Multiplying by $P_k(v)$ on both sides of \eqref{eq:rte} and integrating over $v\in[-1, 1]$, we derive the moment equations:
\begin{equation}
\begin{aligned}
	 \partial_t m_0 + \partial_x m_1 &= -  \sigma_a m_0 \\
	 \partial_t m_1 + \frac{1}{3} \partial_x m_0 + \frac{2}{3} \partial_x m_2  &= - ( \sigma_s +  \sigma_a) m_1 \\
	& \cdots \\
    \partial_t m_{N-1} + \frac{N-1}{2N-1} \partial_x m_{N-2} + \frac{N}{2N-1} \partial_x m_{N}  &= - ( \sigma_s +  \sigma_a) m_{N-1} \\
	 \partial_t m_N + \frac{N}{2N+1} \partial_x m_{N-1} + \frac{N+1}{2N+1} \partial_x m_{N+1}  &= - ( \sigma_s +  \sigma_a) m_N
\end{aligned}
\end{equation}

The above system is clearly not closed due to the existence of $\partial_x m_{N+1}$ in the last equation. The learning gradient approach proposed in \cite{huang2021gradient} is to find  a relation between $\partial_x m_{N+1}$ and the gradients on lower order moments:
\begin{equation}\label{eq:anstaz-learn-gradient}
	\partial_x m_{N+1} = \sum_{i=0}^N \mathcal{N}_i(m_0,m_1,\cdots,m_N) \partial_x m_i
\end{equation}
with $\mathcal{N}=(\mathcal{N}_0, \mathcal{N}_1, \cdots, \mathcal{N}_N):\mathbb{R}^{N+1}\rightarrow\mathbb{R}^{N+1}$ approximated by a neural network and learned from data. Plugging \eqref{eq:anstaz-learn-gradient} into the closure system, we derive the moment closure model:
\begin{equation}\label{eq:closure-model}
\begin{aligned}
	 \partial_t m_0 + \partial_x m_1 &= -  \sigma_a m_0 \\
	 \partial_t m_1 + \frac{1}{3} \partial_x m_0 + \frac{2}{3} \partial_x m_2  &= - ( \sigma_s +  \sigma_a) m_1 \\
	& \cdots \\
    \partial_t m_{N-1} + \frac{N-1}{2N-1} \partial_x m_{N-2} + \frac{N}{2N-1} \partial_x m_{N}  &= - ( \sigma_s +  \sigma_a) m_{N-1} \\
    \partial_t m_N + \frac{N}{2N+1} \partial_x m_{N-1} + \frac{N+1}{2N+1} \brac{\sum_{k=0}^N \mathcal{N}_k(m_0,m_1,\cdots,m_N) \partial_x m_k}  &= - ( \sigma_s +  \sigma_a) m_N.	 
\end{aligned}
\end{equation}
In the numerical tests, this approach is shown to be accurate in the optically thick regime, intermediate regime and the optically thin regime. Moreover, the accuracy of this gradient-based model is much better than the approach based on creating a ML closure directly trained to match the moments, as well as the conventional $P_N$ closure. However, this model exhibits numerical instability due to the loss of hyperbolicity \cite{huang2021gradient}. This severely restricts the application of this model, especially for long time simulations.

\subsection{Hyperbolic ML moment closure}
\label{sec:hmoment-closure}

In this work, our main idea to enforce the hyperbolicity is motivated by the observation that the coefficient matrix of the closure system is a lower Hessenberg matrix. We write the closure model \eqref{eq:closure-model} into an equivalent form:
\begin{equation}
    \partial_t \bm{m} + A \partial_x \bm{m} = S \bm{m}
\end{equation}
with $\bm{m} = (m_0, m_1, \cdots, m_N)^T$ and the coefficient matrix $A\in\mathbb{R}^{(N+1)\times(N+1)}$:
\begin{equation}\label{eq:coefficient-matrix}
	A = 
	\begin{pmatrix}
    0 				& 	1 				&	0  			& 0  	&	\dots 	& 	0 	\\
    \frac{1}{3} 	& 	0 				& \frac{2}{3} 	& 0  	&	\dots 	& 	0	\\
    0 & \frac{2}{5} & 	0 				& \frac{3}{5} 	& \dots & 0 \\
    \vdots & \vdots &   \vdots 			& \ddots & \vdots & \vdots \\
    0 & 0 & 	\dots 				& \frac{N-1}{2N-1} 	& 0 & \frac{N}{2N-1} \\
    a_0 & a_1 & \dots & a_{N-2} & a_{N-1} & a_N
	\end{pmatrix}
\end{equation}
with
\begin{equation}\label{eq:relation-a-N}
	a_j = 
	\left\{
	\begin{aligned}
	& \frac{N+1}{2N+1}\mathcal{N}_j, \quad & j\ne N-1, \\
	& \frac{N}{2N+1} + \frac{N+1}{2N+1}\mathcal{N}_j, \quad & j=N-1.
	\end{aligned}
	\right.
\end{equation}
and the source term
\begin{equation}
    S = \textrm{diag}(-\sigma_a, -(\sigma_s+\sigma_a), \cdots, -(\sigma_s+\sigma_a)).
\end{equation}
In what follows, we will use the properties of the Hessenberg matrix in Section \ref{sec:hessenberg} to analyze the real diagonalizability of the coefficient matrix $A$ in \eqref{eq:coefficient-matrix}.

We first write down the associated polynomial sequence of $A$ using the definition \eqref{eq:recurrence-polynomial}:
\begin{subequations}
\begin{align}
	q_0(x) &= 1, \label{eq:ml-recurrence-1} \\
	\frac{i}{2i-1} q_i(x) &= x q_{i-1}(x) - \frac{i-1}{2i-1} q_{i-2}(x), \quad i=1,\cdots,N \label{eq:ml-recurrence-2} \\
	q_{N+1}(x) &= x q_{N}(x) - \sum_{k=0}^N a_k q_k(x). \label{eq:ml-recurrence-3}
\end{align}	
\end{subequations}
Notice that \eqref{eq:ml-recurrence-2} is exactly the same as the recurrence relation for the Legendre polynomial. Thus, we have
\begin{equation}
	q_i(x) = P_i(x), \quad i=0,1,\cdots,N.
\end{equation}
Then from \eqref{eq:ml-recurrence-3}, we derive
\begin{equation}
	q_{N+1}(x) = \frac{N+1}{2N+1} P_{N+1}(x) + \frac{N}{2N+1} P_{N-1}(x) - \sum_{k=0}^N a_k P_k(x),
\end{equation}
where we used the recurrence relation for the Legendre polynomial:
\begin{equation}
	\frac{N+1}{2N+1} P_{N+1}(x) = x P_N(x) - \frac{N}{2N+1} P_{N-1}(x).
\end{equation}
By Theorem \ref{thm:root-is-eigenvalue}, it is easy to derive the following theorem:
\begin{thm}\label{thm:ml-real-diagonalizable}
	For the coefficient matrix $A$ in \eqref{eq:coefficient-matrix}, the associated polynomial sequence satisfies:
	\begin{subequations}
	\begin{align}
		q_i(x) &= P_i(x), \quad i=0,1,\cdots,N, \\
		q_{N+1}(x) &= \frac{N+1}{2N+1} P_{N+1}(x) + \frac{N}{2N+1} P_{N-1}(x) - \sum_{k=0}^N a_k P_k(x),
	\end{align}
	\end{subequations}
	where $P_n(x)$ denotes the Legendre polynomial of degree $n$.
	If all the roots of $q_{N+1}(x)$ are simple, then the characteristic polynomial of $A$ is:
	\begin{equation}
		\det(x I_{N+1} - A) = \rho q_{N+1}(x) = \rho \brac{\frac{N+1}{2N+1} P_{N+1}(x) + \frac{N}{2N+1} P_{N-1}(x) - \sum_{k=0}^N a_k P_k(x)}
	\end{equation}
	with $\rho = \frac{N!}{(2N-1)!!}$. If further assuming all the roots of $q_{N+1}(x)$ are simple and real, then all the eigenvalues of $A$ are distinct and real. In this case, the moment closure system is strictly hyperbolic. If further assuming all the roots of $q_{N+1}(x)$ are simple, real and lie in the interval $[-1, 1]$, then the moment closure system is strictly hyperbolic with physical characteristic speeds.
\end{thm}
\begin{rem}
	From Theorem \ref{thm:real-diagonalizable}, the condition that all the roots of $q_{N+1}(x)$ are simple and real, is also necessary for the moment closure system to be hyperbolic.
\end{rem}

Next, we will derive the relation between the eigenvalues of $A$ (or the roots of $q_{N+1}(x)$) and the weights of the gradients in \eqref{eq:anstaz-learn-gradient}. In particular, we will represent $\{\mathcal{N}_k \}_{0\le k\le N}$ in \eqref{eq:anstaz-learn-gradient} using the eigenvalues of $A$.

We denote the distinct real eigenvalues of $A$ by $\{r_k\}_{0\le k \le N}$. Then, by Theorem \ref{thm:ml-real-diagonalizable}, we have
\begin{equation}
	(x-r_0)(x-r_1)\cdots(x-r_N) = \rho \brac{\frac{N+1}{2N+1} P_{N+1}(x) + \frac{N}{2N+1} P_{N-1}(x) - \sum_{k=0}^N a_k P_k(x)}.
\end{equation}
First, we expand the characteristic polynomial using a set of monomial basis:
\begin{equation}
	\det(x I_{N+1} - A) = c_0 + c_1 x + \cdots + c_N x^N + x^{N+1}.
\end{equation}
Using Vieta's formulas, we relate the coefficients $\{c_k\}_{0\le k \le N}$ to the sums and products of its roots $\{r_k\}_{0\le k \le N}$:
\begin{equation}\label{eq:transform-r-to-c}
\begin{aligned}
	& r_0 + r_1 + \cdots + r_{N-1} + r_N = - c_N, \\
	& (r_0r_1 + r_0r_2 + \cdots + r_0r_N) + (r_1r_2 + r_1r_3 + \cdots + r_1r_N) + \cdots + r_{N-1}r_N = c_{N-1}, \\
	& \qquad \vdots \\
	& r_0 r_1 \cdots r_{N-1} r_N = (-1)^{N+1} c_0,
\end{aligned}
\end{equation}
or equivalently written as a compact formulation
\begin{equation}
	\sum_{0\le i_1 < i_2 < \cdots < i_k\le N} \brac{\Pi_{j=1}^k r_{i_j} } = (-1)^k c_{N+1-k}, \quad k = 1, 2, \cdots, N+1.
\end{equation}
Here the indices $i_k$ are sorted in strictly increasing order to ensure each product of $k$ roots is used exactly once.

Then, we establish the relationship between $\{c_k\}_{0\le k\le N}$ to $\{a_k\}_{0\le k\le N}$.
Using the generating function of Legendre polynomials, one can express the monomial in terms of a summation of Legendre polynomials \cite{szeg1939orthogonal}. We present the conclusion in the following lemma and include the proof in the appendix.
\begin{lem}\label{lem:monomial-to-legendre}
	For any integer $m\ge 0$, there holds the following equality:
	\begin{equation}\label{eq:monomial-to-legendre-floor}
		x^m = \sum_{k=0}^{\lfloor m/2\rfloor} F(m,k) P_{m-2k}(x),
	\end{equation}
	with $F(m,k) = \frac{m!(2m-4k+1)}{2^k k!(2m-2k+1)!!}$.
	Here $P_n(x)$ is the $n$-th order Legendre polynomial, and $\lfloor \cdot \rfloor$ is the floor function which takes a real number $x$ as input, and gives the greatest integer less than or equal to $x$ as output.	
\end{lem}
We rewrite \eqref{eq:monomial-to-legendre-floor} into an equivalent formulation:
\begin{equation}\label{eq:monomial-to-legendre}
	x^m = \sum_{k=0}^m b_{mk} P_k(x), \quad m\ge 0,
\end{equation}
with
\begin{equation}
	b_{mk}=\left\{
		\begin{aligned}
		& F(m, \frac{1}{2}(m-k)), \quad & \textrm{if} ~ m\equiv k \;(\bmod\; 2), \\
		& 0,  \quad & \textrm{otherwise}.
		\end{aligned}
		\right.
\end{equation}
From this formula, we can expand any polynomial $\sum_{i=0}^n c_i x^i$ in terms of Legendre polynomials:
\begin{equation}
	\sum_{i=0}^n c_i x^i = \sum_{i=0}^n c_i \brac{\sum_{k=0}^i b_{ik} P_k(x)} = \sum_{k=0}^n \brac{\sum_{i=k}^n c_ib_{ik}} P_k(x) = \sum_{k=0}^n \alpha_k P_k(x),
\end{equation}
with 
\begin{equation}
	\alpha_k = \sum_{i=k}^n c_ib_{ik}.
\end{equation}
We apply the above relation to derive the relationship between $\{c_k\}_{0\le k\le N}$ to $\{a_k\}_{0\le k\le N}$:
\begin{equation}
	c_0 + c_1 x + \cdots + c_N x^N + x^{N+1} = \rho \brac{\frac{N+1}{2N+1} P_{N+1}(x) + \frac{N}{2N+1} P_{N-1}(x) - \sum_{k=0}^N a_k P_k(x)},
\end{equation}
and obtain
\begin{subequations}
\begin{align}
	-\rho a_k &= \sum_{i=k}^{N+1} c_i b_{ik}, \quad k = 0,1,\cdots,N-3,N-2,N, \label{eq:monomial-to-legendre-1} \\
	\rho (\frac{N}{2N+1} - a_{N-1}) &= \sum_{i=N-1}^{N+1} c_i b_{i,N-1}, \label{eq:monomial-to-legendre-2} \\
	\rho \frac{N+1}{2N+1} &= \sum_{i=N+1}^{N+1} c_i b_{i,N+1}, \label{eq:monomial-to-legendre-3}
\end{align}
\end{subequations}
with $c_{N+1}:=1$. The last one \eqref{eq:monomial-to-legendre-3} is automatically satisfied since $b_{N+1,N+1}=F(N+1,0)=\frac{(N+1)!(2N+3)}{(2N+3)!!} = \frac{(N+1)!}{(2N+1)!!}$.

Lastly, we rewrite \eqref{eq:monomial-to-legendre-1}-\eqref{eq:monomial-to-legendre-2} in terms of $\{\mathcal{N}_{k}\}_{0\le k\le N}$ using the relation \eqref{eq:relation-a-N}:
\begin{equation}\label{eq:transform-c-to-n}
	\mathcal{N}_k = - \frac{2N+1}{\rho (N+1)}\sum_{i=k}^{N+1} c_i b_{ik}, \quad k = 0,1,\cdots,N,
\end{equation}
with $c_{N+1}:=1$.

Now, together with \eqref{eq:transform-r-to-c} and \eqref{eq:transform-c-to-n}, we have expressed $\{\mathcal{N}_k\}_{0\le k\le N}$ using the eigenvalues $\{r_k\}_{0\le k\le N}$.

\section{Architectures and training of the neural network}\label{sec:nn}

In this section, we provide the architectures and training of the proposed neural networks that enforces the hyperbolicity of the closure system. 

\subsection{Architectures of the neural network}

We start with the first neural network architecture. As shown in Figure \ref{fig:schematic-nn}, this neural network begins with a fully connected neural network denoted by $\mathcal{M}_{\theta}:\mathbb{R}^{N+1}\rightarrow\mathbb{R}^{N+1}$ with the input being the lower order moments $(m_0, m_1, \cdots,m_{N})$ and the output denoted by $(z_0, z_1, \cdots,z_{N})$. Here $\theta$ denotes the collection of all the parameters to be trained in the neural network. It is then followed by a component-wise hyperbolic tangent function to enforce the boundness of the eigenvalues, i.e. $r_i = \tanh(z_i)$ for $i=0,1,\cdots,N$. Lastly, two sublayers representing the Vieta's formula \eqref{eq:transform-r-to-c} and a linear transformation \eqref{eq:transform-c-to-n} are applied to produce the weights $(\mathcal{N}_0, \mathcal{N}_1, \cdots, \mathcal{N}_N)$ in the gradient based closure in \eqref{eq:anstaz-learn-gradient} as the final output. For the ML moment closure model resulted by this neural network in Figure \ref{fig:schematic-nn}, we have the following conclusion:
\begin{thm}\label{thm:nn-bound}
    The ML moment closure model \eqref{eq:closure-model} resulting from the neural network with bounded eigenvalues shown in Figure \ref{fig:schematic-nn} is weakly hyperbolic. Moreover, it guarantees the physical characteristic speeds, i.e.,  the eigenvalues lie in the interval $[-1, 1]$.
\end{thm}

\begin{figure}
    \centering
    \includegraphics[width=0.9\textwidth, clip=true, trim=0mm 75mm 0mm 35mm]{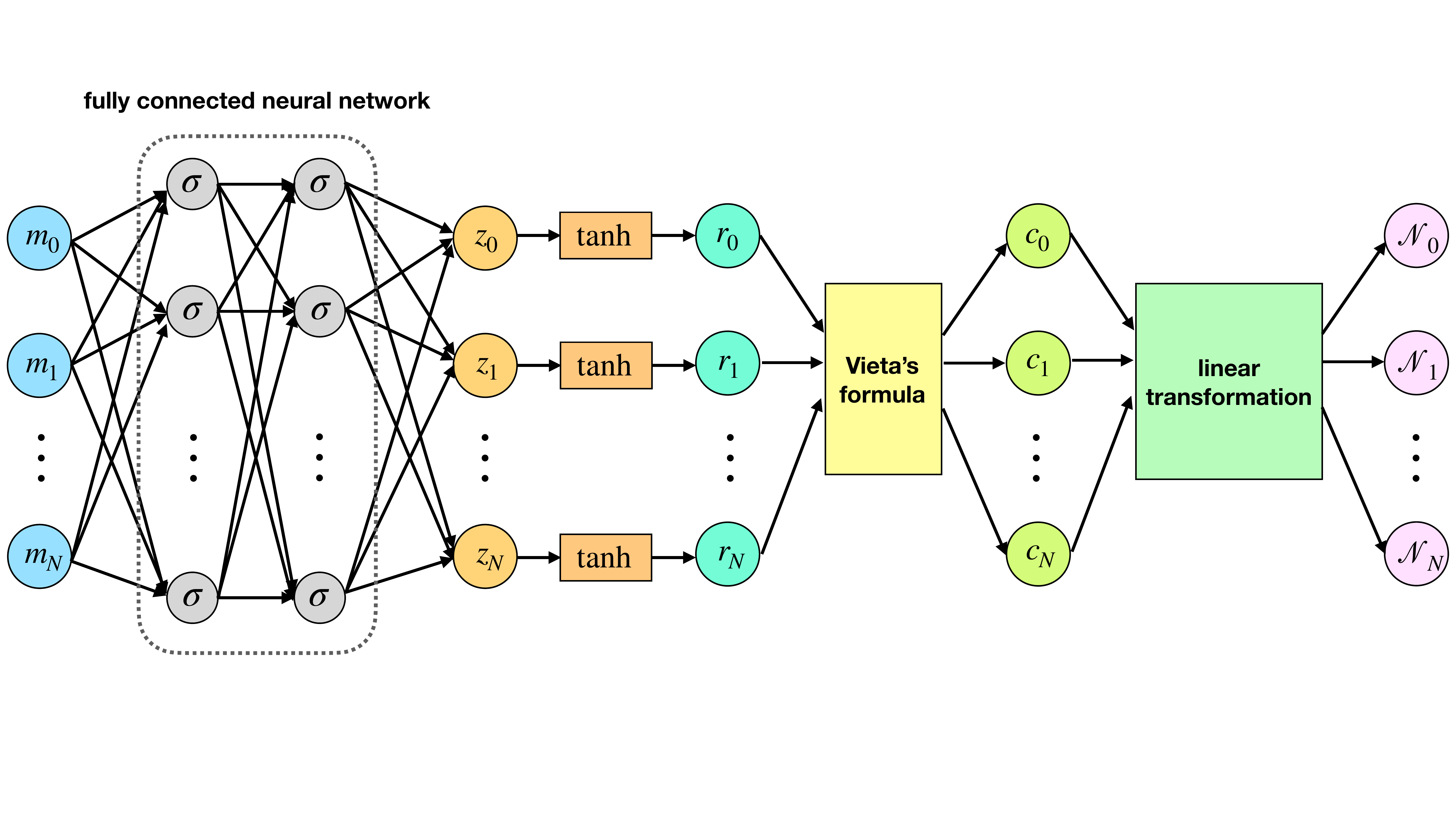}
    \caption{
    Schematic of the neural network with bounded eigenvalues. Input: the moments $(m_0, m_1, \cdots, m_N)$, output: the weights $(\mathcal{N}_0, \mathcal{N}_1, \cdots, \mathcal{N}_N)$ in the gradient based closure in \eqref{eq:anstaz-learn-gradient}.
    The Vieta's formula is given in \eqref{eq:transform-r-to-c}. The linear transformation from $(c_0, c_1, \cdots, c_N)$ to $(\mathcal{N}_0, \mathcal{N}_1, \cdots, \mathcal{N}_N)$ is given in \eqref{eq:transform-c-to-n}.
    }
    \label{fig:schematic-nn}.
\end{figure}

There is a small gap between the implementation of the neural network in Figure \ref{fig:schematic-nn} and the theory presented in previous sections. From Theorem \ref{thm:real-diagonalizable}, all the eigenvalues being distinct and real is a necessary and sufficient condition for the moment closure system to be hyperbolic. However, we do not force all the eigenvalues to be distinct in the current neural network architecture. Thus, this only guarantees that the resulting system is theoretically weakly hyperbolic instead of the system being hyperbolic, and might cause an instability issue.

In the numerical tests in Section \ref{sec:numerical-test}, we observe that the ML closure model is numerically stable for $N\ge6$. Meanwhile, for $N\le5$, the model has some stability issues, and these stability issues are associated with the case when two of the eigenvalues are within $10^{-3}$ of each other on a range of grid points, see the detailed discussion in Figure \ref{fig:const-distinct-instability} in Section \ref{sec:numerical-test}. 

To fix this problem, we tried several approaches by enforcing that the eigenvalues are distinct (or well separated). The first approach is to divide the interval $[-1, 1]$ into $(N+1)$ uniform subintervals with some threshold gap between two neighbouring subintervals: $I_k = [-1 + \frac{2k}{N+1} + \gamma, -1 + \frac{2(k+1)}{N+1} - \gamma]$ for $k=0,\cdots,N$. Then, we put exactly one eigenvalue into each subinterval, enforced by a scaled hyperbolic tangent function. Here $\gamma\ge0$ is a small number to guarantee a minimum distance of any two eigenvalues. We take $\gamma=0$ and $10^{-3}$ in the implementation. This approach is motivated by the fact that, in the $P_N$ closure, each subinterval contains exactly one eigenvalue. However, we find that, the neural network results in large training errors in the training process, 
which are generally larger than 14\% with  $N=3,4,\cdots,10$.
In these tests, we fix the number of nodes to be 64 and the number of layers to be 6. 
This indicates that the assumption of the uniform distribution of the eigenvalues is too restrictive, so that the approximation power of the neural network is not enough to produce an accurate closure. We will not focus on this neural network in the afterwards.

The other approach is to replace the hyperbolic tangent layer in Figure \ref{fig:schematic-nn} with some other postprocessing layers. As illustrated in Figure \ref{fig:schematic-nn-distinct}, we first applied some positive function, $\kappa$, to the outputs of the fully-connected neural network except for the first component:
\begin{equation}
    \tilde{z}_0 = z_0, \qquad \tilde{z}_i = \kappa(z_i), \quad i=1,\cdots,N,
\end{equation}
Here $\kappa=\kappa(x) \ge \gamma > 0$ is a strictly positive function taken as
\begin{equation}
    \kappa(x) = \ln(1 + e^x) + \gamma
\end{equation}
with $\gamma=0.1$.
Then, it is followed by a linear transformation:
\begin{equation}\label{eq:r-z-linear-transform}
    r_i = \sum_{k=0}^i \tilde{z}_k, \quad i=0,\cdots,N,
\end{equation}
which produces the eigenvalues of the closure system. Next, the Vieta's formula and the linear transformation are imposed as in Figure \ref{fig:schematic-nn}. This approach can guarantee that the eigenvalues are distinct, i.e. $r_0 < r_1 < \cdots < r_N$, but may lose the boundness property of the eigenvalues. Nevertheless, in the numerical simulations, we observe that the model has the physical characteristic speeds for most of the time although this constrain is not enforced explicitly, see the discussion in Figure \ref{fig:const-long-time} in Section \ref{sec:numerical-test}. For the ML moment closure model resulted from this neural network in Figure \ref{fig:schematic-nn-distinct}, we have the following conclusion:
\begin{thm}\label{thm:nn-distinct}
    The ML moment closure model \eqref{eq:closure-model} resulted from the neural network with distinct eigenvalues shown in Figure \ref{fig:schematic-nn-distinct} is strictly hyperbolic.
\end{thm}

\begin{figure}
    \centering
    \includegraphics[width=1.0\textwidth, clip=true, trim=0mm 90mm 20mm 40mm]{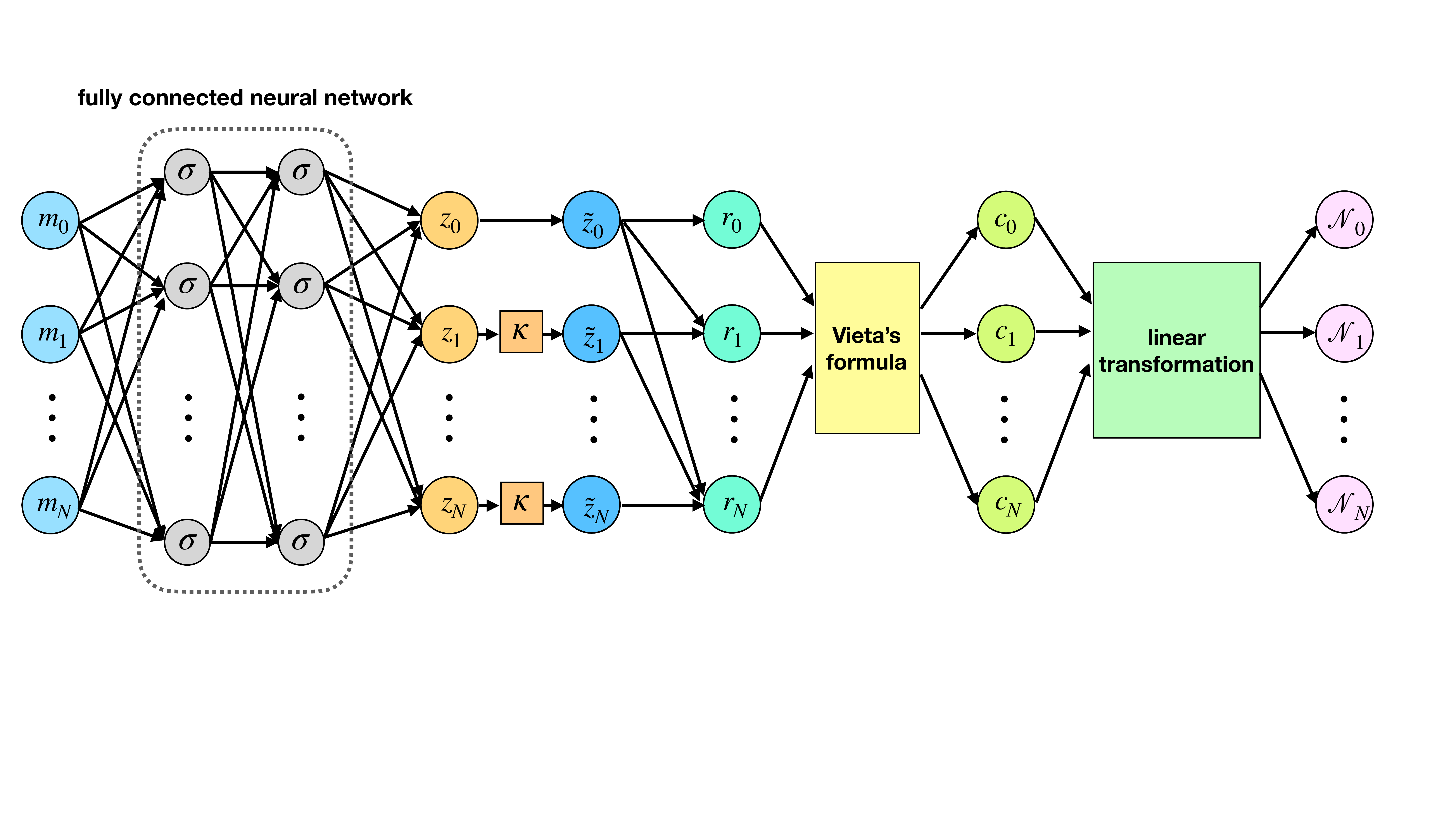}
    \caption{
    Schematic of the neural network with distinct eigenvalues. Input: the moments $(m_0, m_1, \cdots, m_N)$, output: the weights $(\mathcal{N}_0, \mathcal{N}_1, \cdots, \mathcal{N}_N)$ in the gradient based closure in \eqref{eq:anstaz-learn-gradient}. Here, $\kappa=\kappa(x)>0$ is a positive function. The layer connecting $(\tilde{z}_0, \tilde{z}_1, \cdots, \tilde{z}_N)$ and $(r_0, r_1, \cdots, r_N)$ is given in \eqref{eq:r-z-linear-transform}. The Vieta's formula is given in \eqref{eq:transform-r-to-c}. The linear transformation from $(c_0, c_1, \cdots, c_N)$ to $(\mathcal{N}_0, \mathcal{N}_1, \cdots, \mathcal{N}_N)$ is given in \eqref{eq:transform-c-to-n}.
    }
    \label{fig:schematic-nn-distinct}.
\end{figure}

We remark that the the current neural network architectures could not guarantee (strongly) hyperbolicity and the physical characteristic speeds simultaneously. We also tried to enforce the two desired properties by adding some penalty terms into the loss function, but we did not get satisfactory results. Nevertheless, we will more fully explore this direction in our future work.

\subsection{Training of the neural network}

For the training of the neural network, we take 1000 total epochs (the number of iterations in the optimization process). We investigated two activation functions including the hyperbolic tangent ($\tanh$) function and Rectified Linear Unit (ReLU) function. The learning rate is set to be $10^{-3}$ in the initial epoch and decays by 0.5 every 100 epochs. The $L^2$ regularization is applied with weight $10^{-7}$. The batch size is taken to be 1024. The training is implemented within the PyTorch framework \cite{paszke2019pytorch}. We use the same hyperparameters for the two neural networks.

Following \cite{huang2021gradient}, in the training process, the loss function is taken to be:
\begin{equation}
    \mathcal{L} = \frac{1}{N_{\textrm{data}}}\sum_{j,n}\abs{{\partial_x m^{\textrm{true}}_{N+1}(x_j, t_n) - \partial_x m^{\textrm{appx}}_{N+1}(x_j, t_n)}}^2.
\end{equation}
Here, $\partial_x m^{\textrm{true}}_{N+1}(x_j, t_n)$ denotes the spatial derivative of $(N+1)$-th order moment at $x=x_j$ and $t=t_n$ computed from the kinetic solver and $\partial_x m^{\textrm{appx}}_{N+1}(x_j, t_n)$ comes from the evaluation of the neural network using \eqref{eq:anstaz-learn-gradient}.

Following \cite{huang2021gradient}, the training data comes from numerically solving the RTE using the space-time discontinuous Galerkin (DG) method \cite{crockatt2017arbitrary,crockatt2019hybrid} with a range of initial conditions in the form of truncated Fourier series and different scattering and absorption coefficients which are constants over the computational domain, see the details in \cite{huang2021gradient}. We train the neural network with 100 different initial data sets. For each initial data set, we run the numerical solver up to $t=1$. The other parameters are the same as in \cite{huang2021gradient}.

To evaluate the accuracy in the training process, we define the relative $L^2$ error for the gradient  to be
\begin{equation}
	E_2 = \sqrt{ \frac{\sum_{j,n}({\partial_x m^{\textrm{true}}_{N+1}(x_j, t_n) - \partial_x m^{\textrm{appx}}_{N+1}(x_j, t_n)})^2}{\sum_{j,n}({\partial_x m^{\textrm{true}}_{N+1}(x_j, t_n)})^2} }.
\end{equation}

The depth and width of neural networks (i.e., the number of hidden layers and the number of nodes in the hidden layers) are crucial hyperparameters in a neural network. Here, we perform a grid search to find the optimal hyperparameters of the neural network including the number of layers and the number of nodes in the first fully-connected neural network $\mathcal{M}_{\theta}$. In particular, we take the number of layers to be $\{2, 3, \cdots, 10\}$ and the number of nodes to be $\{16, 32, \cdots, 256\}$. For the first neural network in Figure \ref{fig:schematic-nn}, the relative $L^2$ errors in the training data with different depths and widths, and $\tanh$ and ReLU activation functions are shown in Figure \ref{fig:training}. Here, we only show the cases with the number of moments to be $N=5, 7, 9$, the cases with $N=6,8$ are similar. With the ReLU activation function, the error decreases when we increase the number of layers and nodes in hidden layers until it saturates, see Figure \ref{fig:training} (b) and Figure \ref{fig:training} (d). However, we observe a different phenomenon with the hyperbolic tangent activation function. When we increase the depth and width, the error decreases only when the network stay relatively small widths 16, 32 and 64 in Figure \ref{fig:training} (a) and widths 16 and 32 in Figure \ref{fig:training} (c). When the neural networks get deeper, the error increases with width. This numerical observation is similar to the well-known vanishing gradient problem. In our current setup, the problem is probably caused by the strong nonlinearity of the Vieta's formula after the fully connected neural network, which stops the neural network from further training. The hyperbolic tangent function, as the activation function, has gradients in the range of $(0, 1)$, which makes it easy for the  neural network  to become stuck in a local minimum due to the vanishing gradient problem. ReLU suffers less from the vanishing gradient problem than the hyperbolic tangent function, because it only saturates in one direction, the one with negative inputs. Other solutions to the vanishing gradient problem, such as residual neural networks (ResNet) \cite{he2016residual} and batch normalization \cite{ioffe2015batch}, may also be applied here to achieve better performance. We will explore this direction in our future work. Moreover, these tests indicate that taking number of layers to be 6 and number of nodes to be 64 and ReLU activation function are good hyperparameters for our neural network. As such these are the values used in all the numerical tests in Section \ref{sec:numerical-test} unless otherwise stated.
\begin{figure}
    \centering
    \subfigure[$\tanh$, $N = 5$]{
    \begin{minipage}[b]{0.46\textwidth}
    \includegraphics[width=1\textwidth]{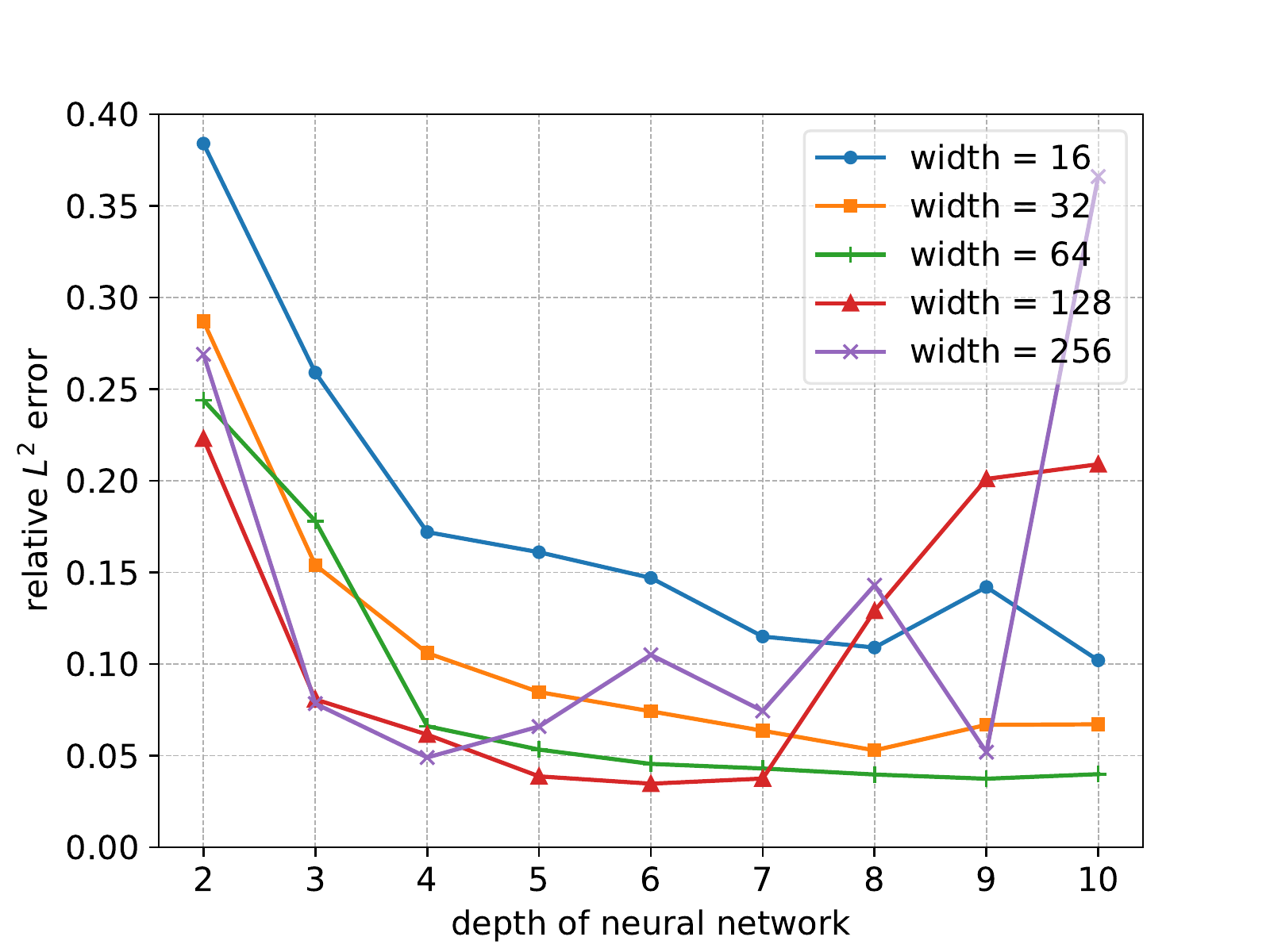}
    \end{minipage}
    }
    \subfigure[ReLU, $N = 5$]{
    \begin{minipage}[b]{0.46\textwidth}    
    \includegraphics[width=1\textwidth]{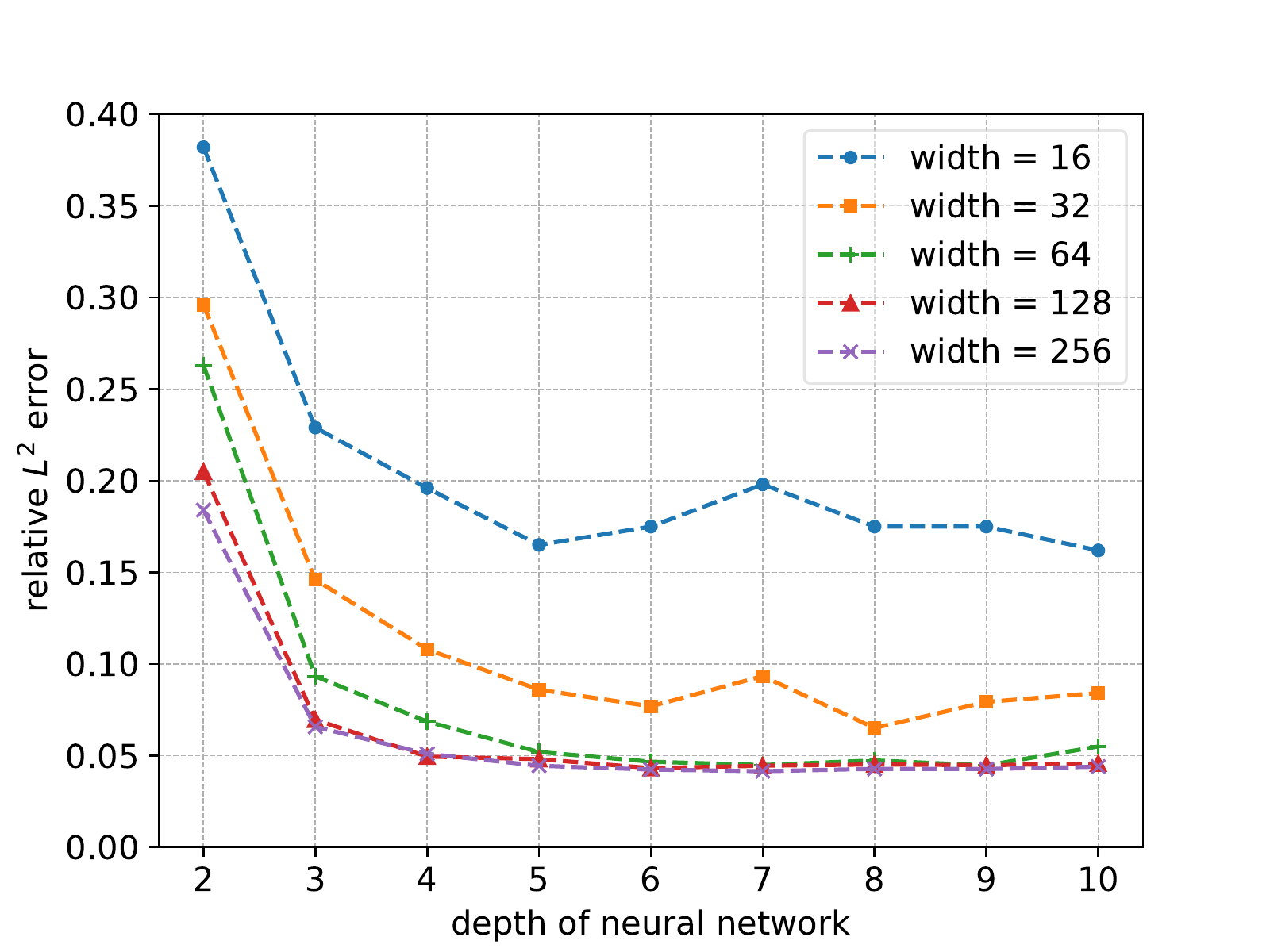}
    \end{minipage}
    }
    \bigskip
    \subfigure[$\tanh$, $N = 7$]{
    \begin{minipage}[b]{0.46\textwidth}
    \includegraphics[width=1\textwidth]{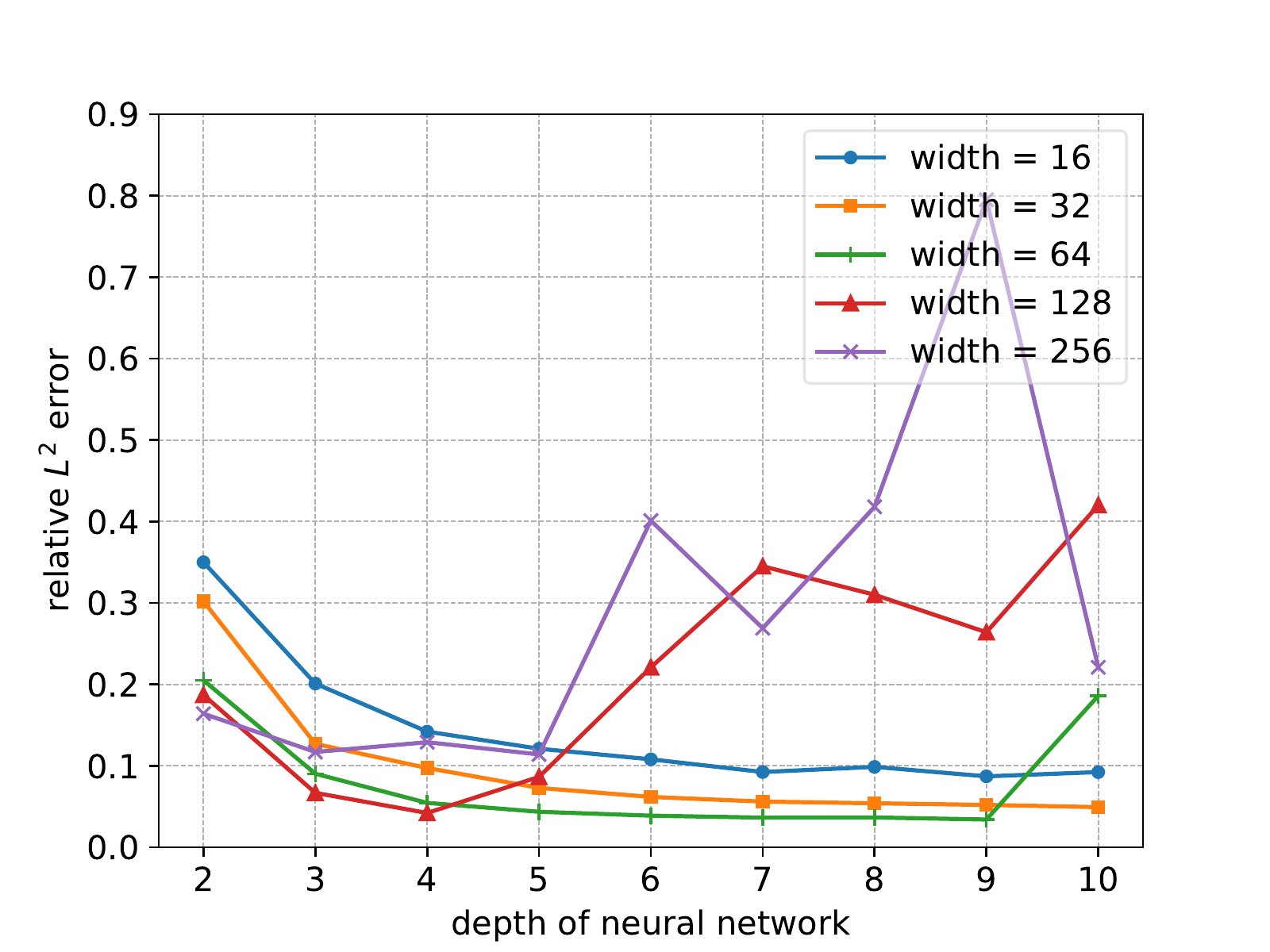}
    \end{minipage}
    }
    \subfigure[ReLU, $N = 7$]{
    \begin{minipage}[b]{0.46\textwidth}    
    \includegraphics[width=1\textwidth]{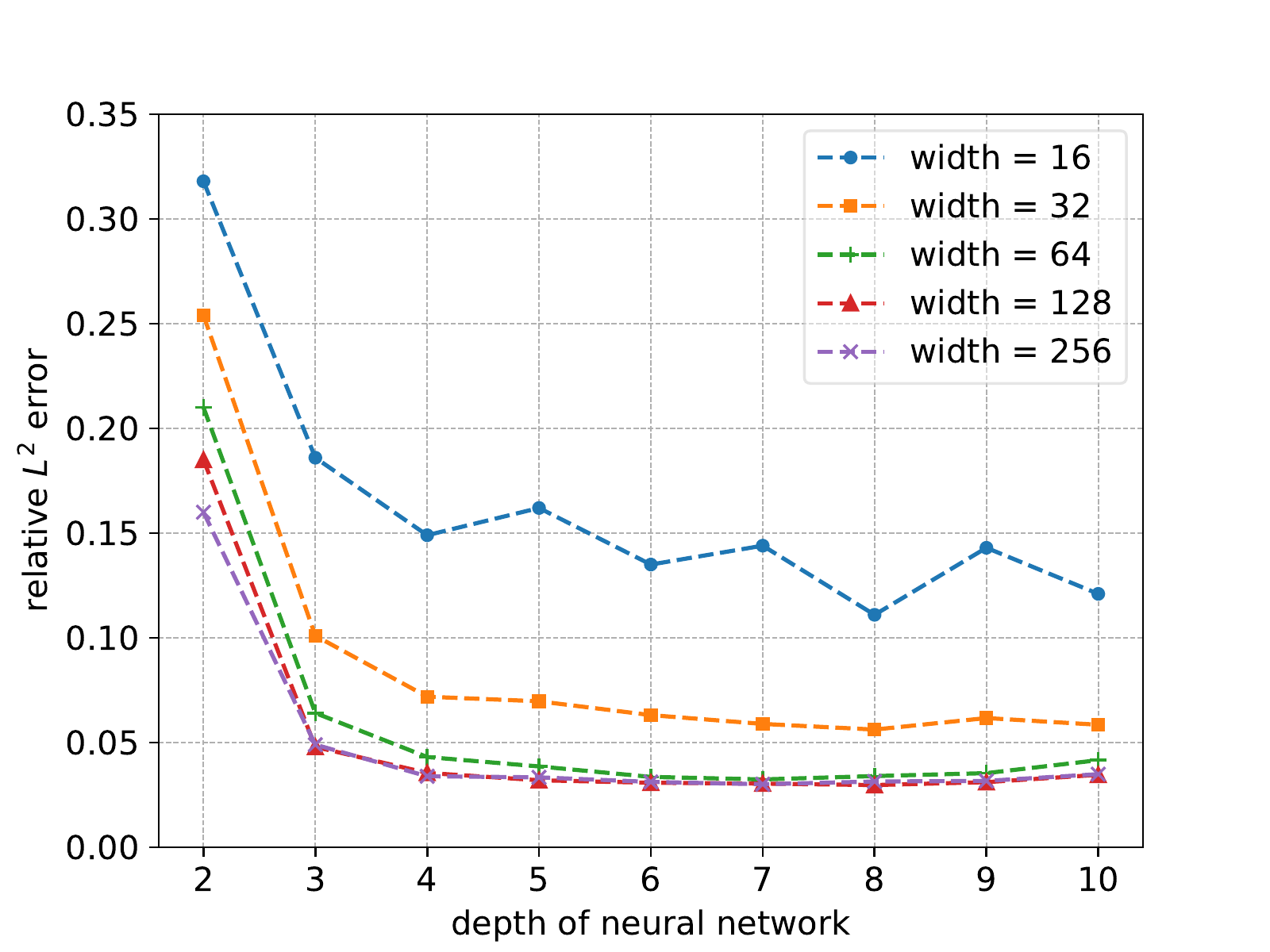}
    \end{minipage}
    }
    \bigskip
    \subfigure[$\tanh$, $N = 9$]{
    \begin{minipage}[b]{0.46\textwidth}
    \includegraphics[width=1\textwidth]{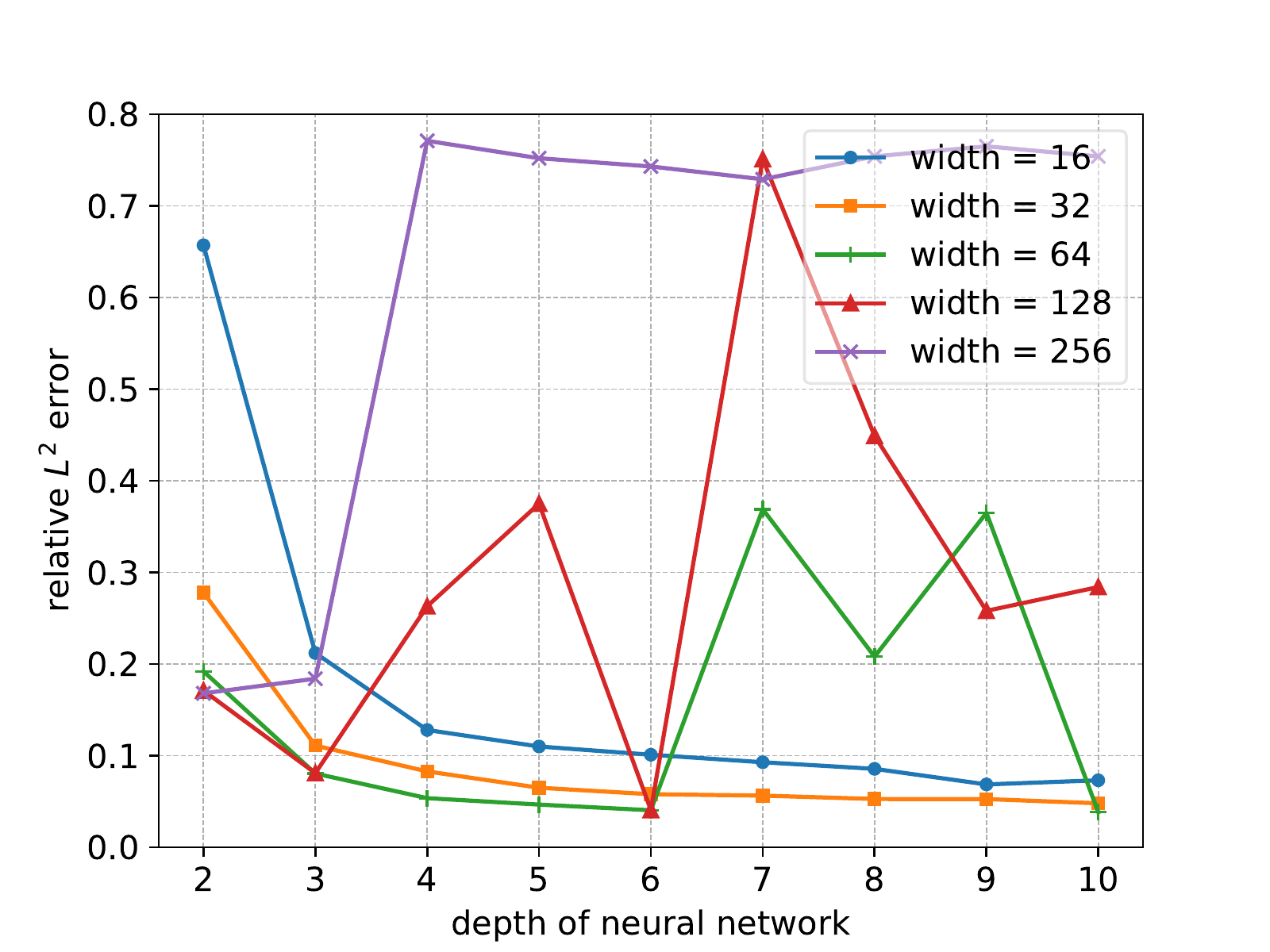}
    \end{minipage}
    }
    \subfigure[ReLU, $N = 9$]{
    \begin{minipage}[b]{0.46\textwidth}    
    \includegraphics[width=1\textwidth]{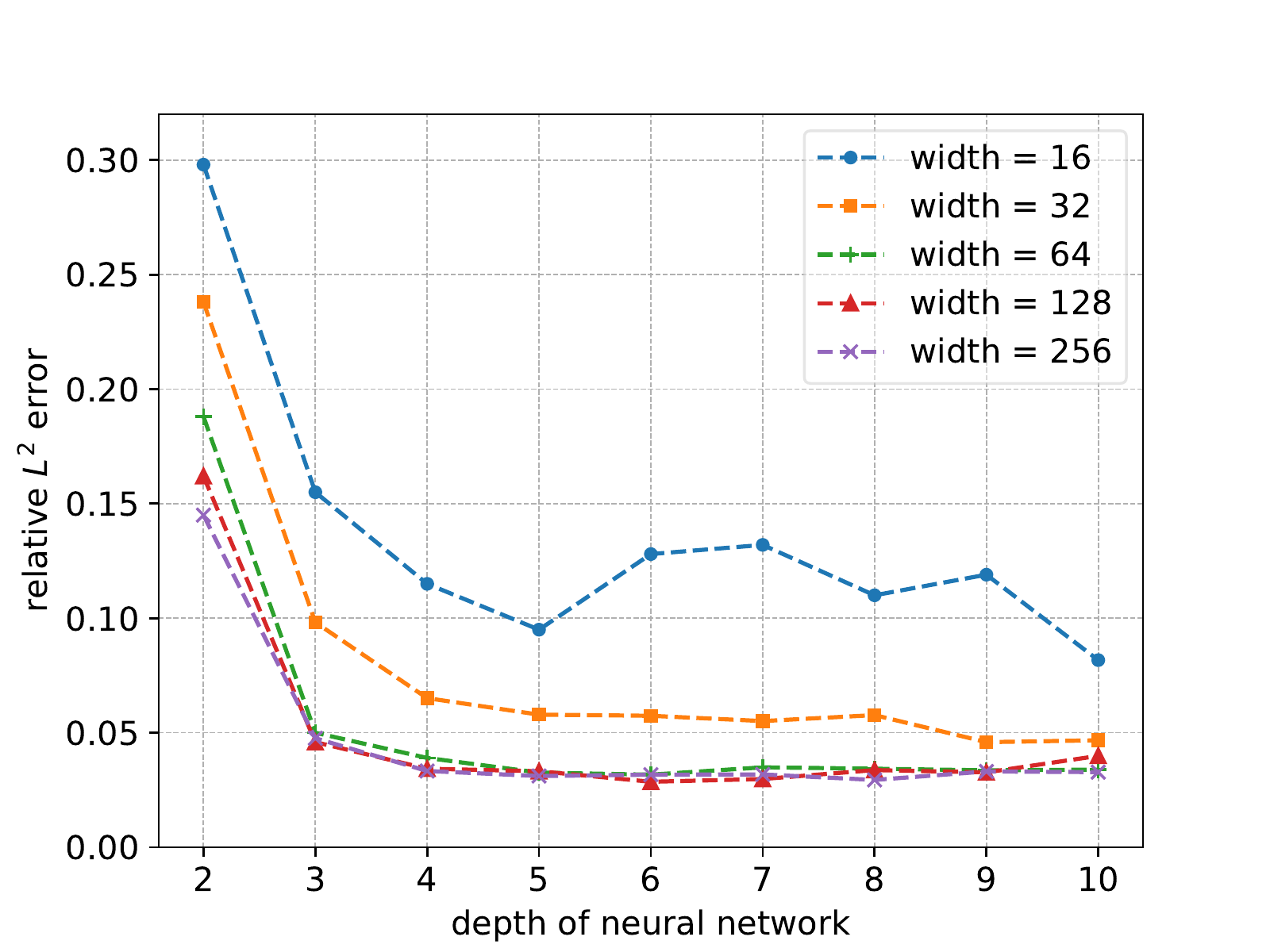}
    \end{minipage}
    }    
	\caption{Relative $L^2$ error in the training data with different depths and widths of the neural networks. Here, we use the first neural network architecture in Figure \ref{fig:schematic-nn}. The number of layers: $2,3,\cdots,10$; the number of nodes in the hidden layers: $16,32,\cdots,256$. Left: hyperbolic tangent activation function; right: ReLU activation function. The number of moments $N=5,7,9$.}
    \label{fig:training}
\end{figure}

\section{Numerical tests}\label{sec:numerical-test}

In this section, we show the performance of our ML closure model on a variety of benchmark tests, including problems with constant scattering and absorption coefficients, Gaussian source problems and two-material problems. The main focus of the tests is on the comparison of four moment closure models: (i) the symmetrizer based hyperbolic ML closure  \cite{huang2021hyperbolic} (termed as ``hyperbolic (symmetrizer)''); (ii) the hyperbolic ML closure with bounded eigenvalues (termed as ``hyperbolic (bound)''), see the neural network architechture in Figure \ref{fig:schematic-nn}; (iii) the hyperbolic ML closure with distinct eigenvalues (termed as ``hyperbolic (distinct)''), see the neural network architechture in Figure \ref{fig:schematic-nn-distinct}; (iv) the classical $P_N$ closure \cite{chandrasekhar1944radiative}.

In all the numerical examples, we take the physical domain to be the unit interval $[0,1]$ and periodic boundary conditions are imposed. To numerically solve the moment closure system, we apply the fifth-order finite difference WENO scheme \cite{jiang1996efficient} with a Lax–Friedrichs flux splitting for the spatial discretization, and the third-order strong-stability-preserving Runge-Kutta (RK) scheme \cite{shu1988efficient} for the time discretization. We take the grid number in space to be $N_x = 256$. The CFL condition is taken to be $\Delta t = 0.8\Delta x/c$ with $c$ being the maximum eigenvalues in all the grid points.

\begin{exam}[constant scattering and absorption coefficients]\label{ex:const}
	The setup of this example is the same as the data preparation. The scattering and absorption coefficients are taken to be constants over the domain. The initial condition is taken to be a truncated Fourier series, see the details in \cite{huang2021gradient}.
	
	In Figure \ref{fig:const-N6}, we show the numerical solutions of $m_0$ and $m_1$ with seven moments in the closure system ($N=6$) in the optically thin regime ($\sigma_s=\sigma_a=1$). It is observed that, at $t=0.5$ and $t=1$, all the hyperbolic ML moment closures agree well the RTE. As a comparison, the $P_N$ closure has large deviations from the exact solution at both $t=0.5$ and $t=1$.
	\begin{figure}
	    \centering
	    \subfigure[$m_0$ at $t=0.5$]{
	    \begin{minipage}[b]{0.46\textwidth}
	    \includegraphics[width=1\textwidth]{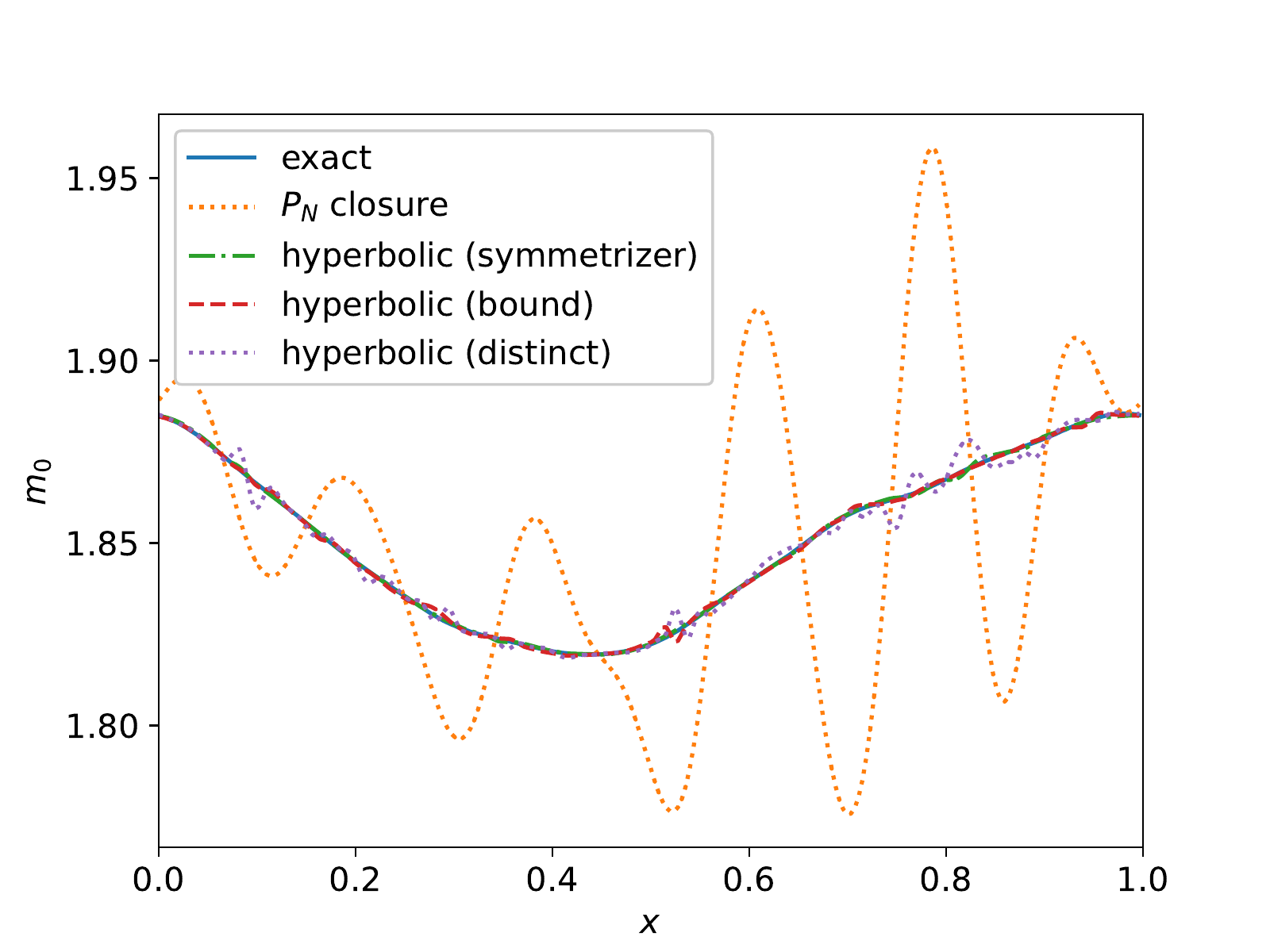}
	    \end{minipage}
	    }
	    \subfigure[$m_1$ at $t=0.5$]{
	    \begin{minipage}[b]{0.46\textwidth}    
	    \includegraphics[width=1\textwidth]{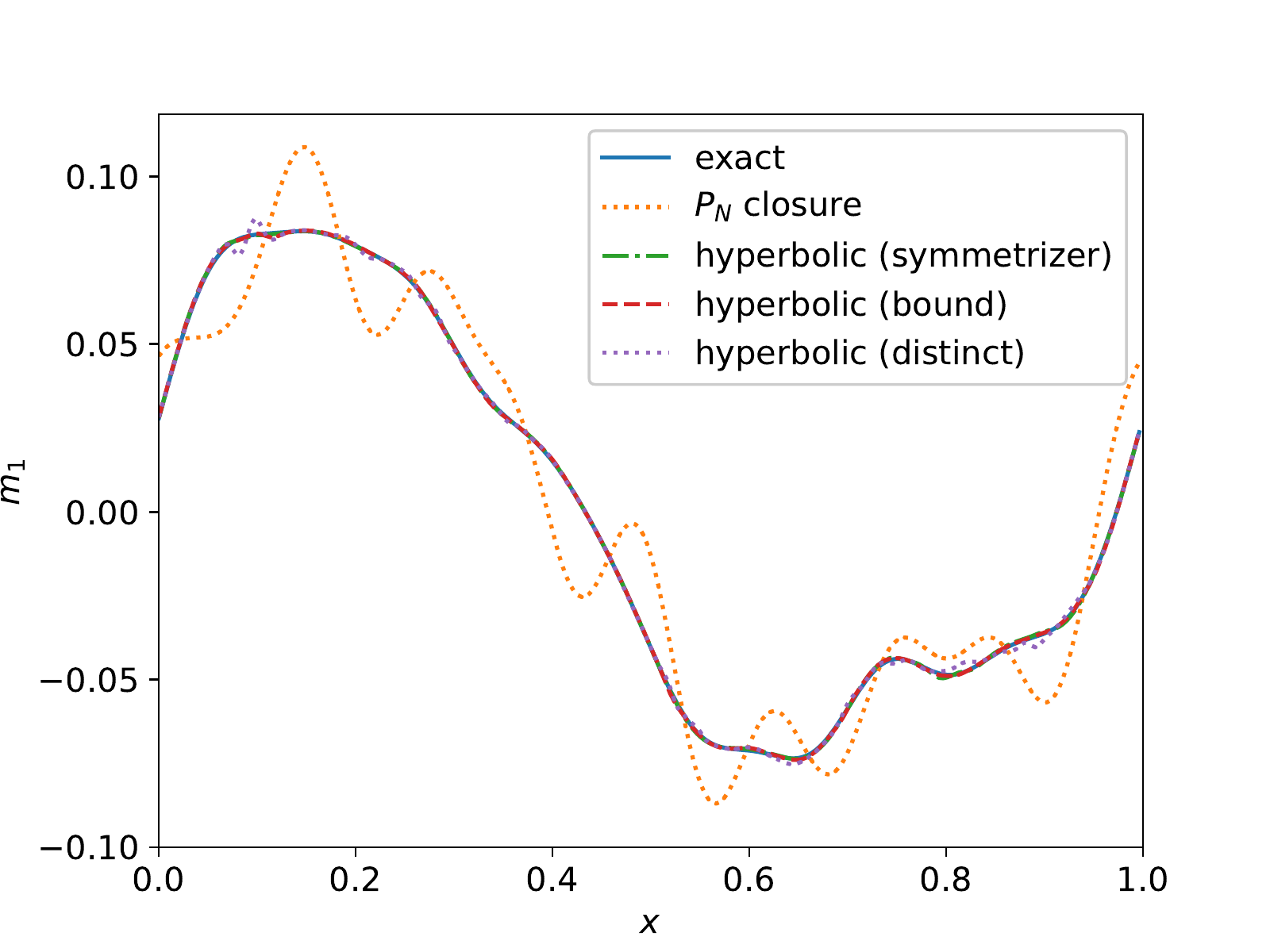}
	    \end{minipage}
	    }
	    \bigskip
	    \subfigure[$m_0$ at $t=1$]{
	    \begin{minipage}[b]{0.46\textwidth}
	    \includegraphics[width=1\textwidth]{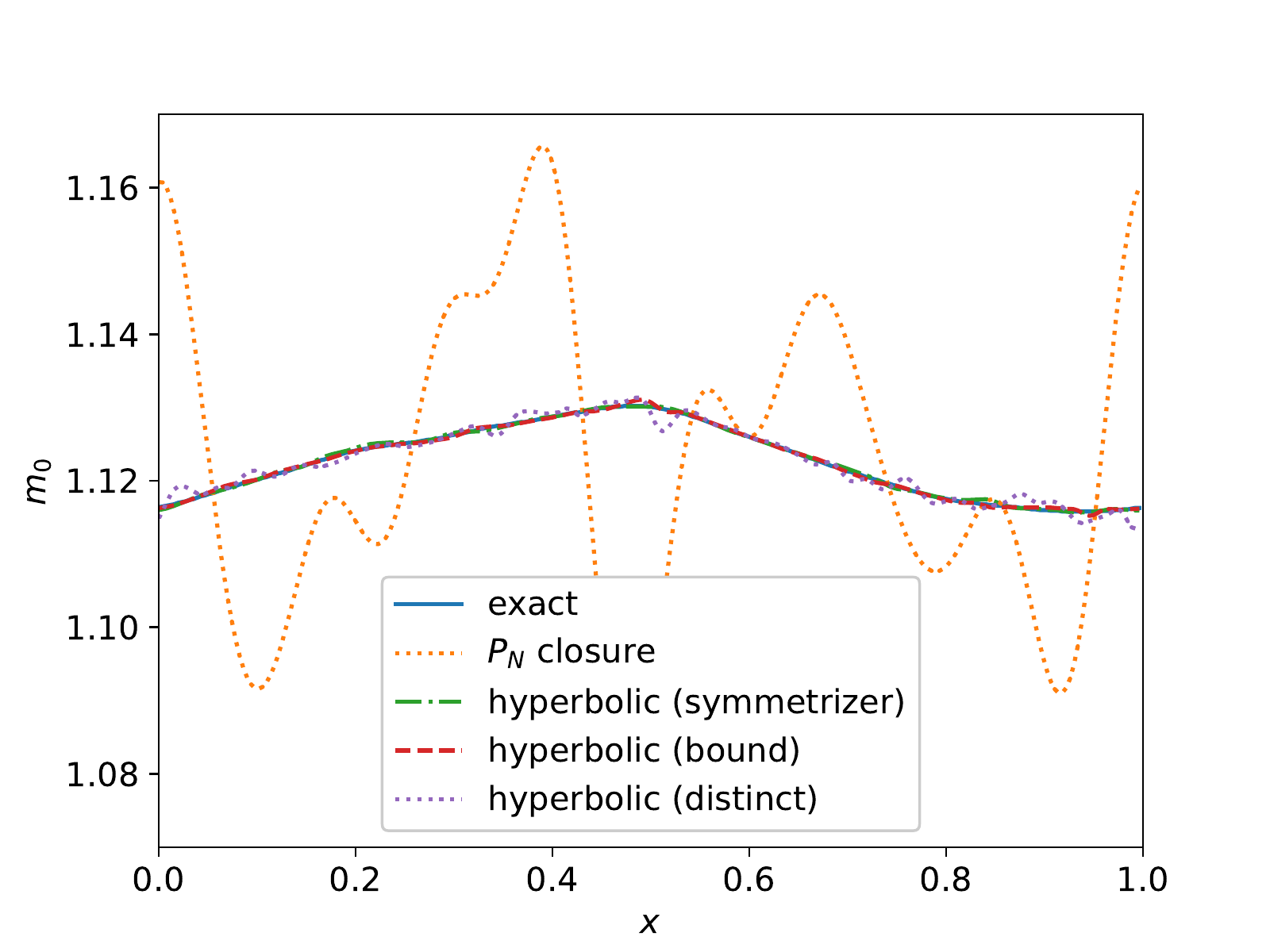}
	    \end{minipage}
	    }
	    \subfigure[$m_1$ at $t=1$]{
	    \begin{minipage}[b]{0.46\textwidth}    
	    \includegraphics[width=1\textwidth]{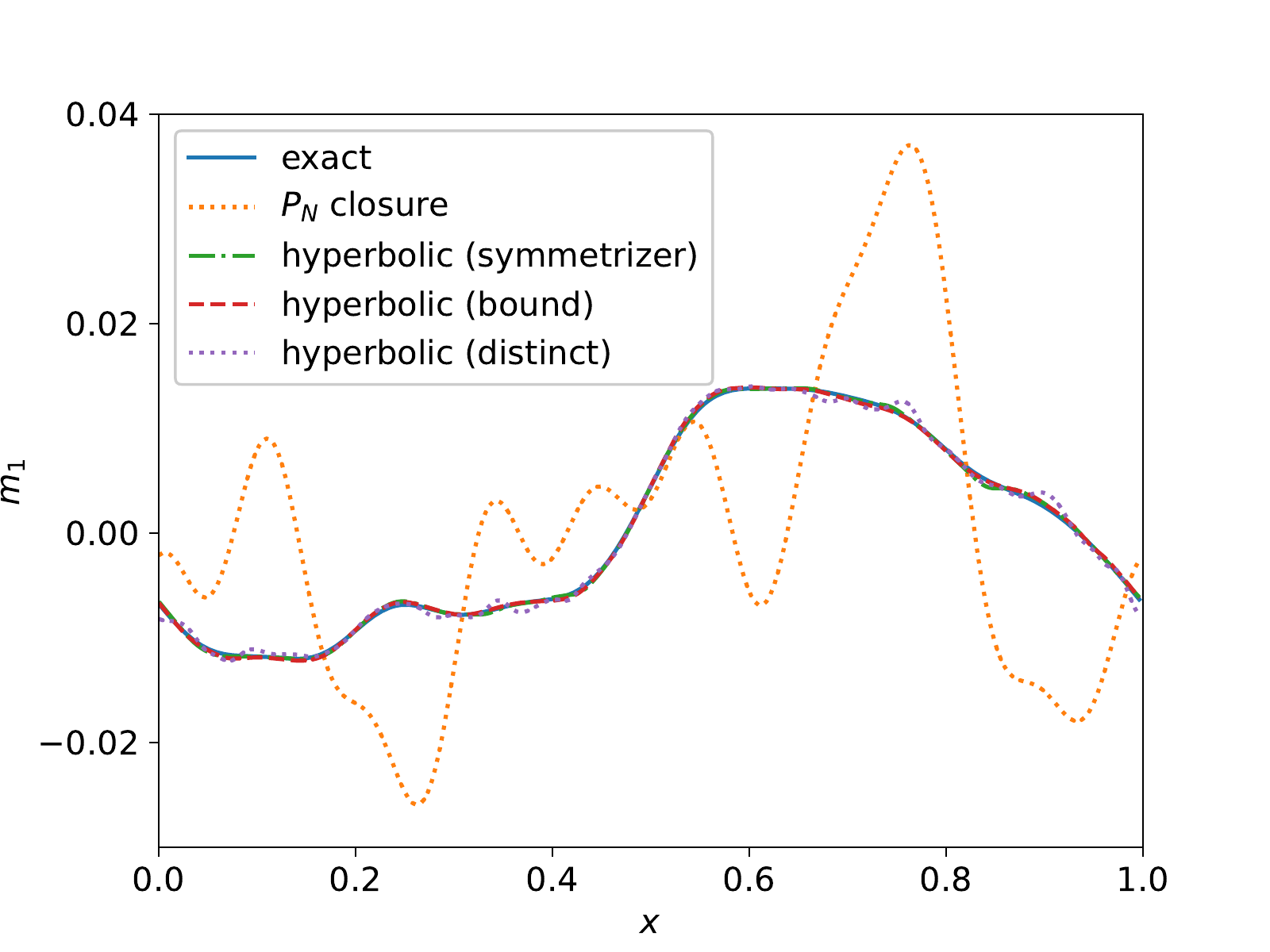}
	    \end{minipage}
	    }
		\caption{Example \ref{ex:const}: constant scattering and absorption coefficients, optically thin regime ($\sigma_s=\sigma_a=1$), $N=6$, $t=0.5$ and $t=1$.}
	    \label{fig:const-N6}
	\end{figure}

	In Figure \ref{fig:const-error-scatter}, we display the log-log scatter plots of the relative $L^2$ error versus the scattering coefficient for $N = 6$ at $t=1$. We observe that, all the hyperbolic ML closures have better accuracy than the $P_N$ closure. Moreover, in the optically thick regime, all the closures perform well. It is also observed that the ML hyperbolic closure model with bounded eigenvalues generally has better accuracy than the other two ML closures.
	\begin{figure}
	    \centering
	    \subfigure[$m_0$ at $t=1$]{
	    \begin{minipage}[b]{0.46\textwidth}
	    \includegraphics[width=1\textwidth]{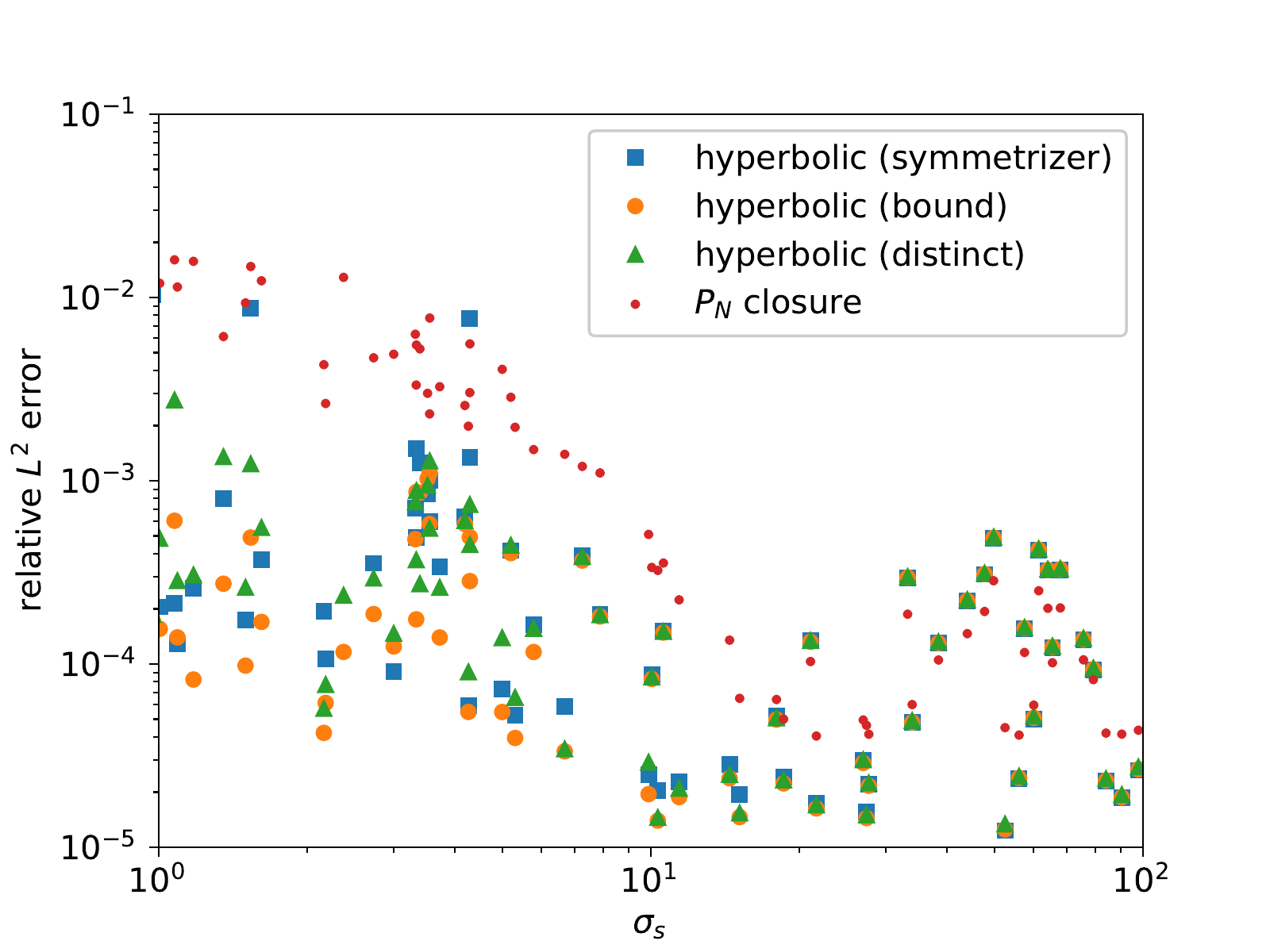}
	    \end{minipage}
	    }
	    \subfigure[$m_1$ at $t=1$]{
	    \begin{minipage}[b]{0.46\textwidth}    
	    \includegraphics[width=1\textwidth]{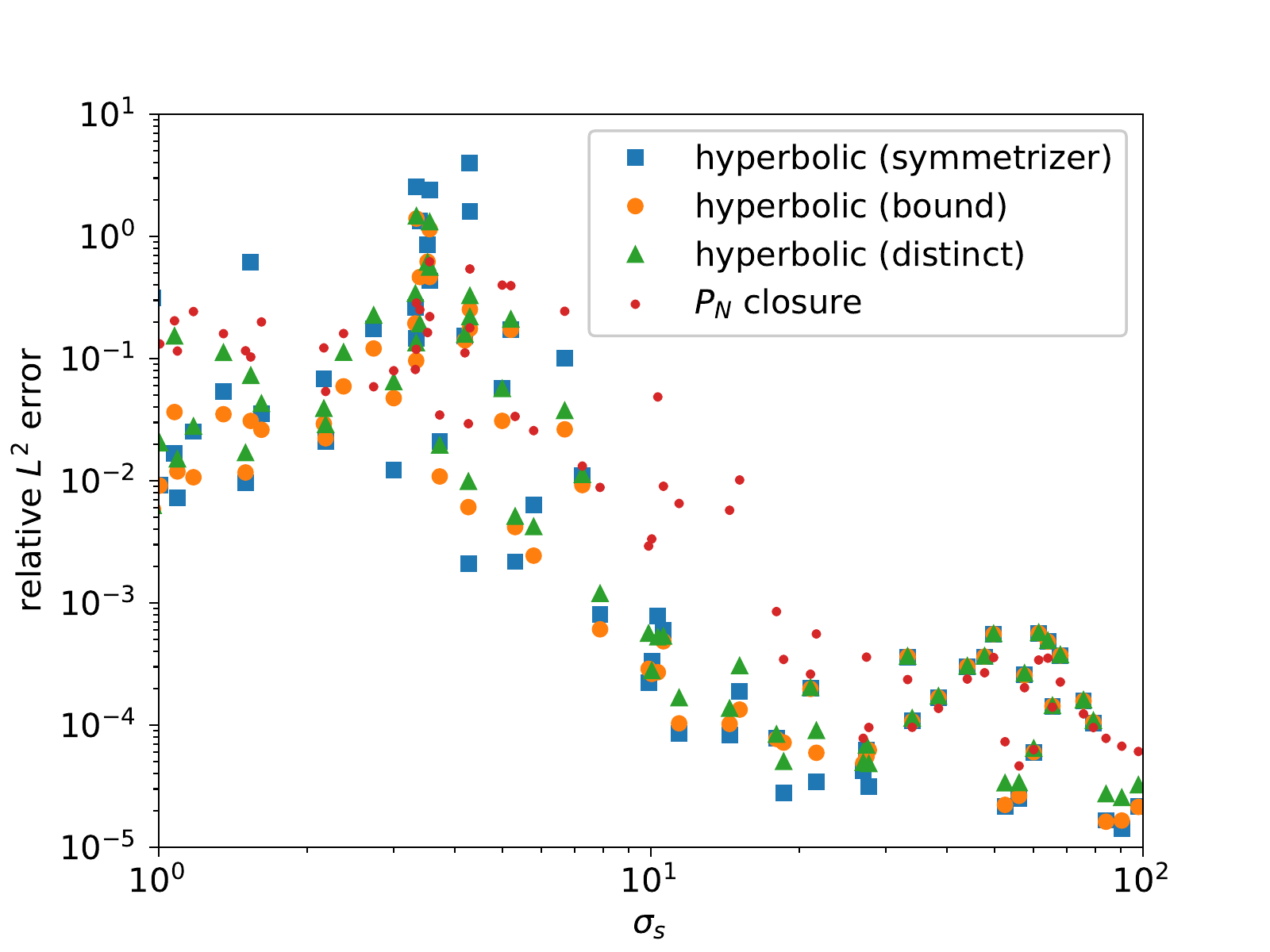}
	    \end{minipage}
	    }	    
		\caption{Example \ref{ex:const}: constant scattering and absorption coefficients, $N=6$ and $t=1$.}
	    \label{fig:const-error-scatter}
	\end{figure}

	In Figure \ref{fig:const-long-time} (a), we present the $L^2$ errors as a function of time for the solutions of the three hyperbolic ML moment closure systems and the solution generated by the RTE in the optically thin regime ($\sigma_s=\sigma_a=1$). We observe that the three hyperbolic closures generate good predictions in the long time simulation up to $t=10$. Moreover, the eigenvalue based ML hyperbolic closure models are more accurate than the symmetrizer based model in \cite{huang2021hyperbolic}. This is probably due to the fact that there is only 4 degrees of freedom in \cite{huang2021hyperbolic}. In contrast, the current eigenvalue based approach makes full use of all the degrees of freedom, which results in better approximation results.
	
	We also display the maximum eigenvalues of the three hyperbolic ML closure models at all the grid points during the time evolution in Figure \ref{fig:const-long-time} (b). It is observed that the eigenvalues are always real numbers, which validates the hyperbolicity feature of the closure models. Moreover, the ML closure with bounded eigenvalues always has physical characteristic speeds bounded by 1. For the closure with distinct eigenvalues, it is interesting to see that the model has the physical characteristic speeds for most of the time although this constrain is not enforced explicitly. The largest eigenvalues of this model during the time evolution is 1.12, which is slightly larger than 1. As a comparision, the symmetrizer based closure in \cite{huang2021hyperbolic} usually violates the physical characteristic speeds, which can be as large as 5.05. The physical characteristic speed of the current ML closure model results in larger time step size in the numerical simulations and thus less computational cost. Moreover, to determine the time step size in the symmetrized based ML model \cite{huang2021hyperbolic} during the time evolution based on the CFL condition, it is required to first compute the coefficient matrix for the closure models and then compute the maximum eigenvalues, which results in additional computational cost. Therefore, the current two ML closure models are better than the symmetrizer based model in \cite{huang2021hyperbolic} in terms of the efficiency.
	\begin{figure}
	    \centering
	    \subfigure[$L^2$ errors of $m_0$ and $m_1$]{
	    \begin{minipage}[b]{0.46\textwidth}
	    \includegraphics[width=1\textwidth]{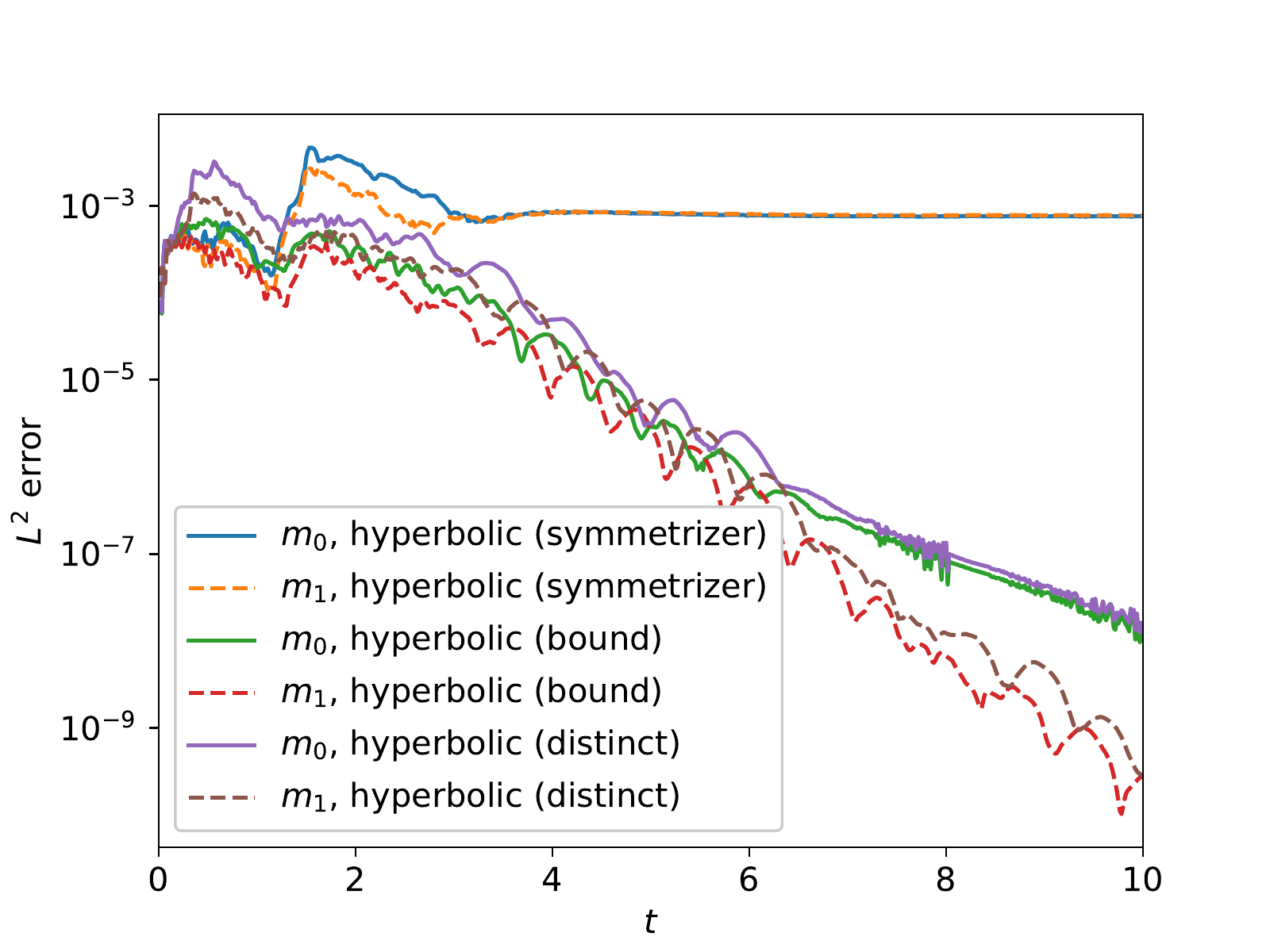}
	    \end{minipage}
	    }
	    \subfigure[maximum eigenvalues]{
	    \begin{minipage}[b]{0.46\textwidth}    
	    \includegraphics[width=1\textwidth]{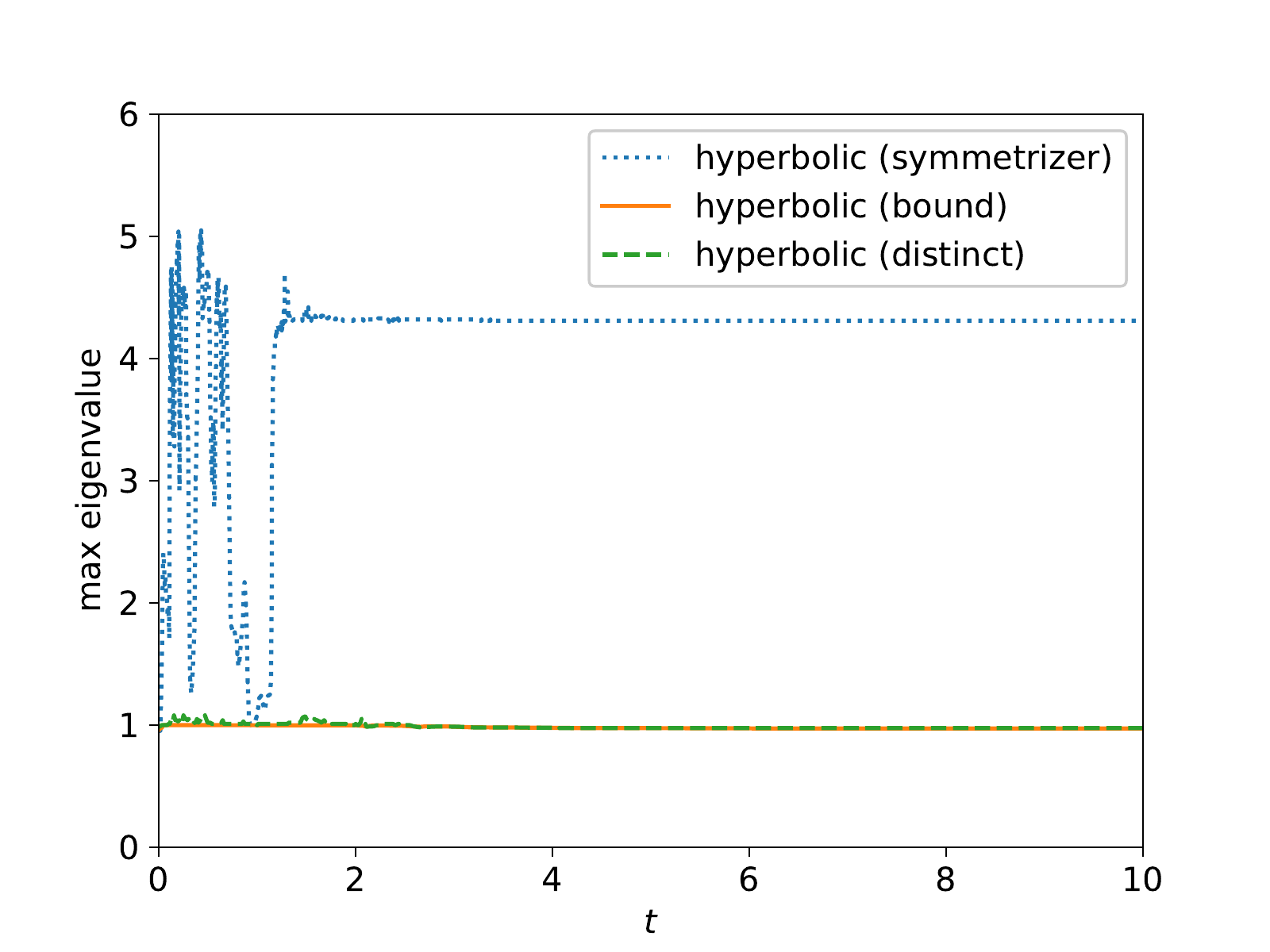}
	    \end{minipage}
	    }	    
		\caption{Example \ref{ex:const}: constant scattering and absorption coefficients, $N=6$.}
	    \label{fig:const-long-time}
	\end{figure}

Next, we discuss the instability issue of the hyperbolic ML closure with bounded eigenvalues. The two eigenvalues get too close for small numbers of moments ($N=3,4,5$), which behaves as if the system is weakly hyperbolic. We simulate the ML closure model with bounded eigenvalues with $N=3$ and $N=5$ in the optically thin regime ($\sigma_s=\sigma_a=1$). The numerical solutions blow up at $t=0.18$ for $N=3$ and $t=1.25$ for $N=5$, see Figure \ref{fig:const-distinct-instability} (b) and Figure \ref{fig:const-distinct-instability} (d) for the $L^{\infty}$ norm of the numerical solutions during the time evolution. As a comparison, the solution stays stable for $N=7$, see Figure \ref{fig:const-distinct-instability} (f).  To investigate this phenomenon in detail, in each time step, we compute the eigenvalues at each grid point, and compute the number of grid points with two eigenvalues which are closer than a given thresholds $\varepsilon$. The number of grid points with close eigenvalues with different thresholds in the time evolution are presented in Figure \ref{fig:const-distinct-instability} (a) and Figure \ref{fig:const-distinct-instability} (c). From the figure, we observe that there are no grid points with close eigenvalues in the beginning. As time evolves, more grid points with non-distinct eigenvalues appear for $N=3$ and $N=5$. For $N=7$, there only exists a couple of grid points with the thresholds $10^{-3}$ and $10^{-4}$ and no grid points with the thresholds $10^{-5}$ and $10^{-6}$. This does not affect the numerical stability of the simulation. 
	\begin{figure}
	    \centering
	    \subfigure[number of grid points with close eigenvalues, $N=3$]{
	    \begin{minipage}[b]{0.46\textwidth}
	    \includegraphics[width=1\textwidth]{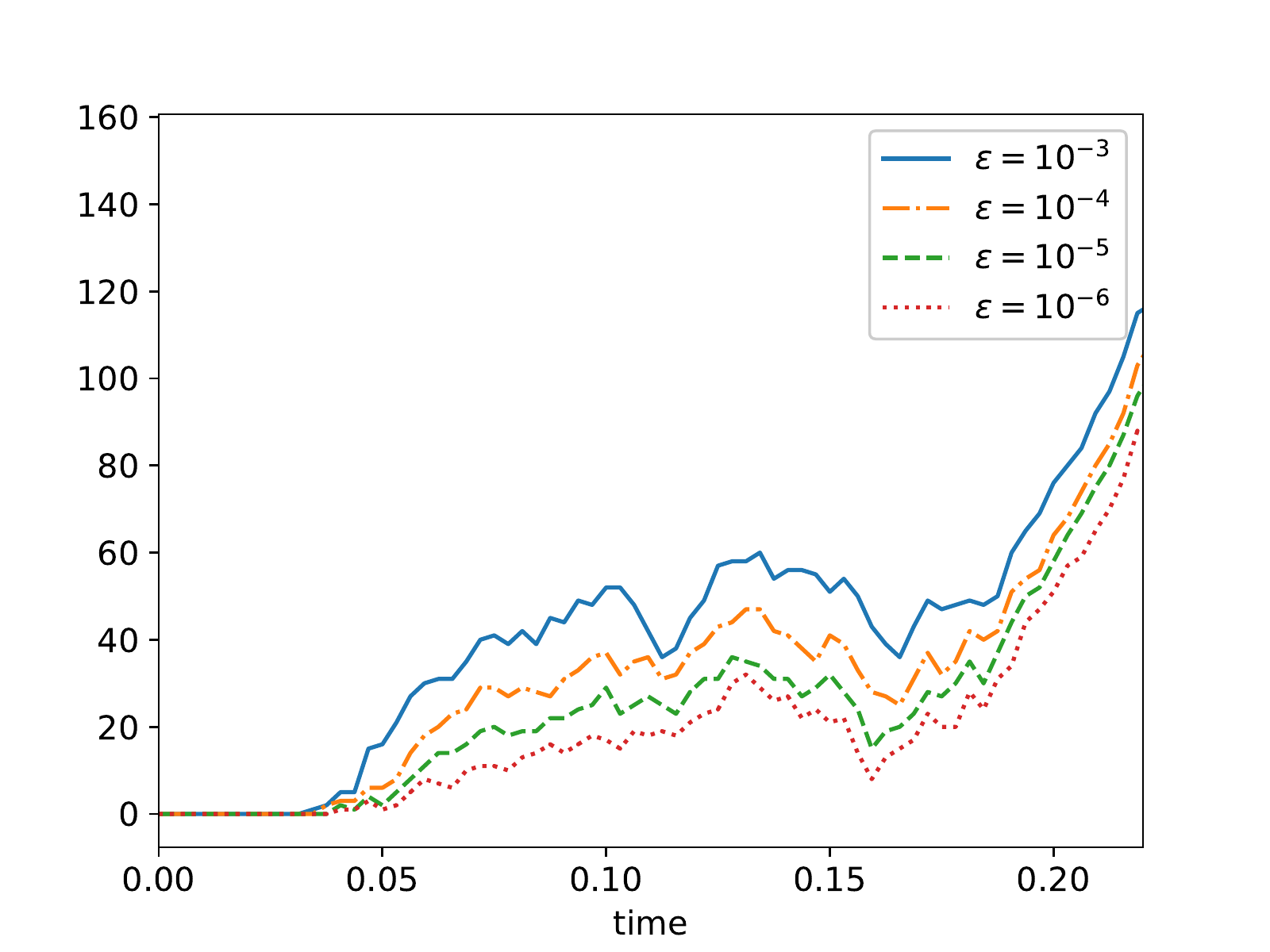}
	    \end{minipage}
	    }
	    \subfigure[$L^{\infty}$ norm of numerical solution, $N=3$]{
	    \begin{minipage}[b]{0.46\textwidth}    
	    \includegraphics[width=1\textwidth]{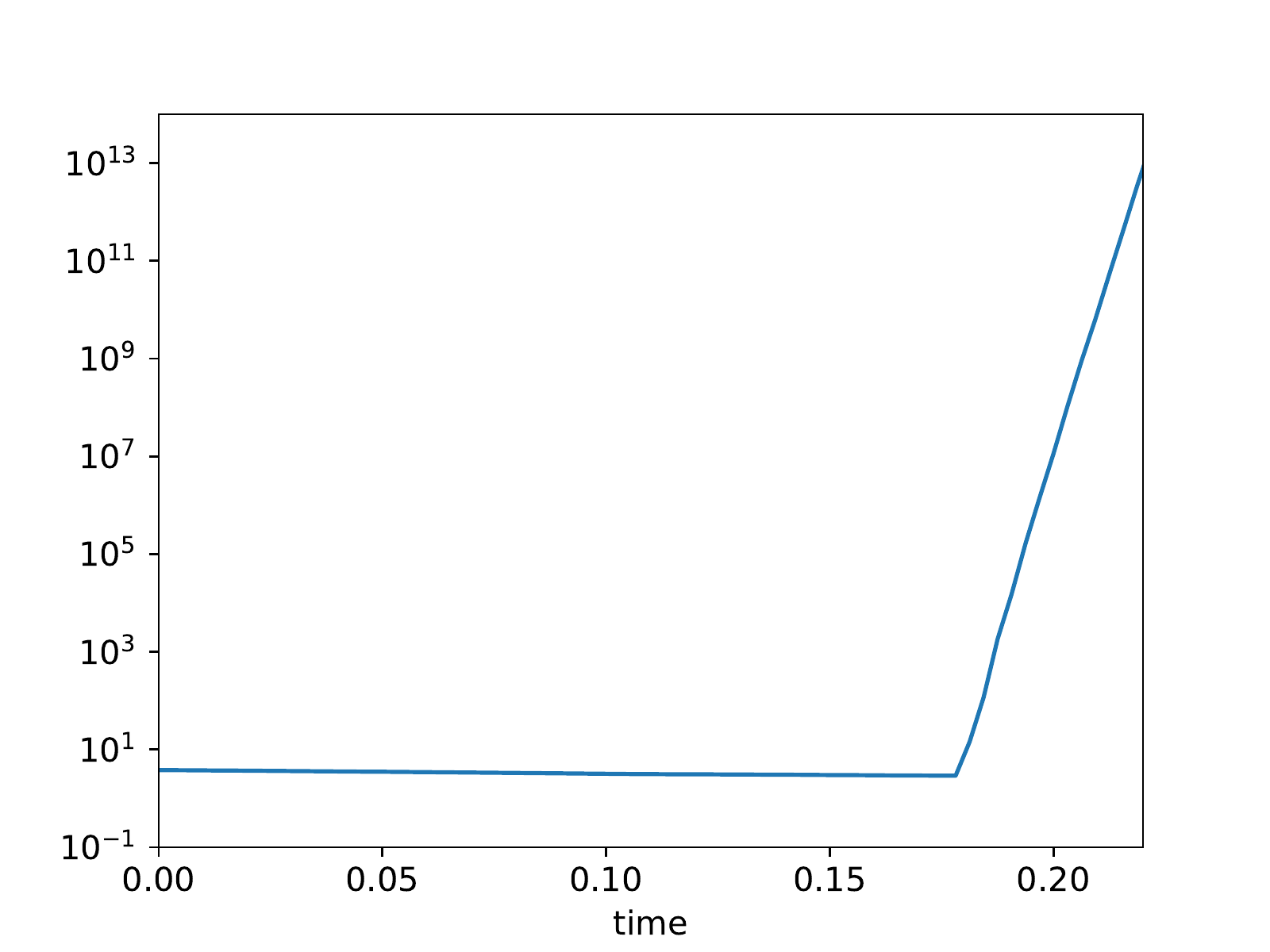}
	    \end{minipage}
	    }
	    \bigskip
	    \subfigure[number of grid points with close eigenvalues, $N=5$]{
	    \begin{minipage}[b]{0.46\textwidth}
	    \includegraphics[width=1\textwidth]{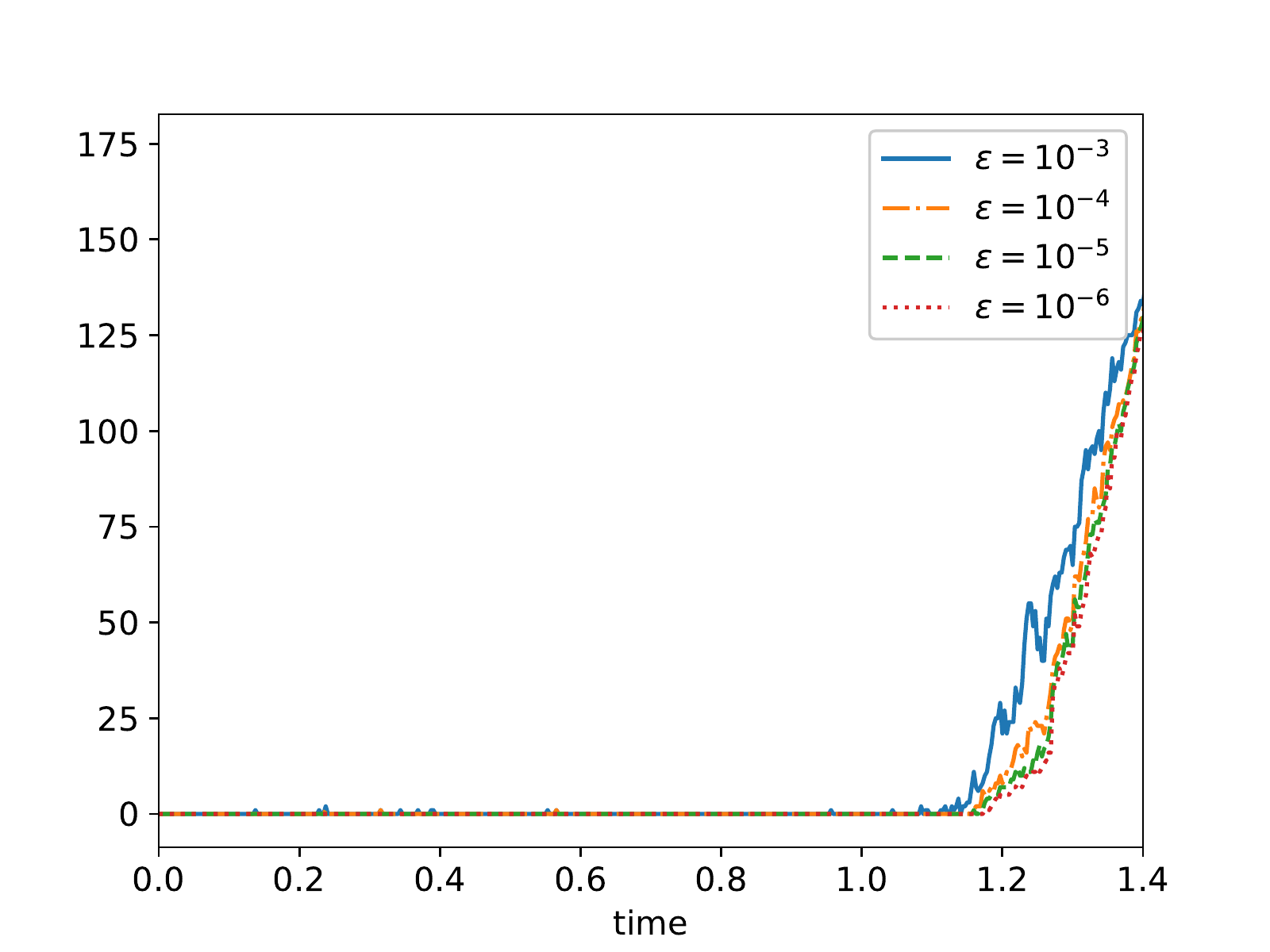}
	    \end{minipage}
	    }
	    \subfigure[$L^{\infty}$ norm of numerical solution, $N=5$]{
	    \begin{minipage}[b]{0.46\textwidth}    
	    \includegraphics[width=1\textwidth]{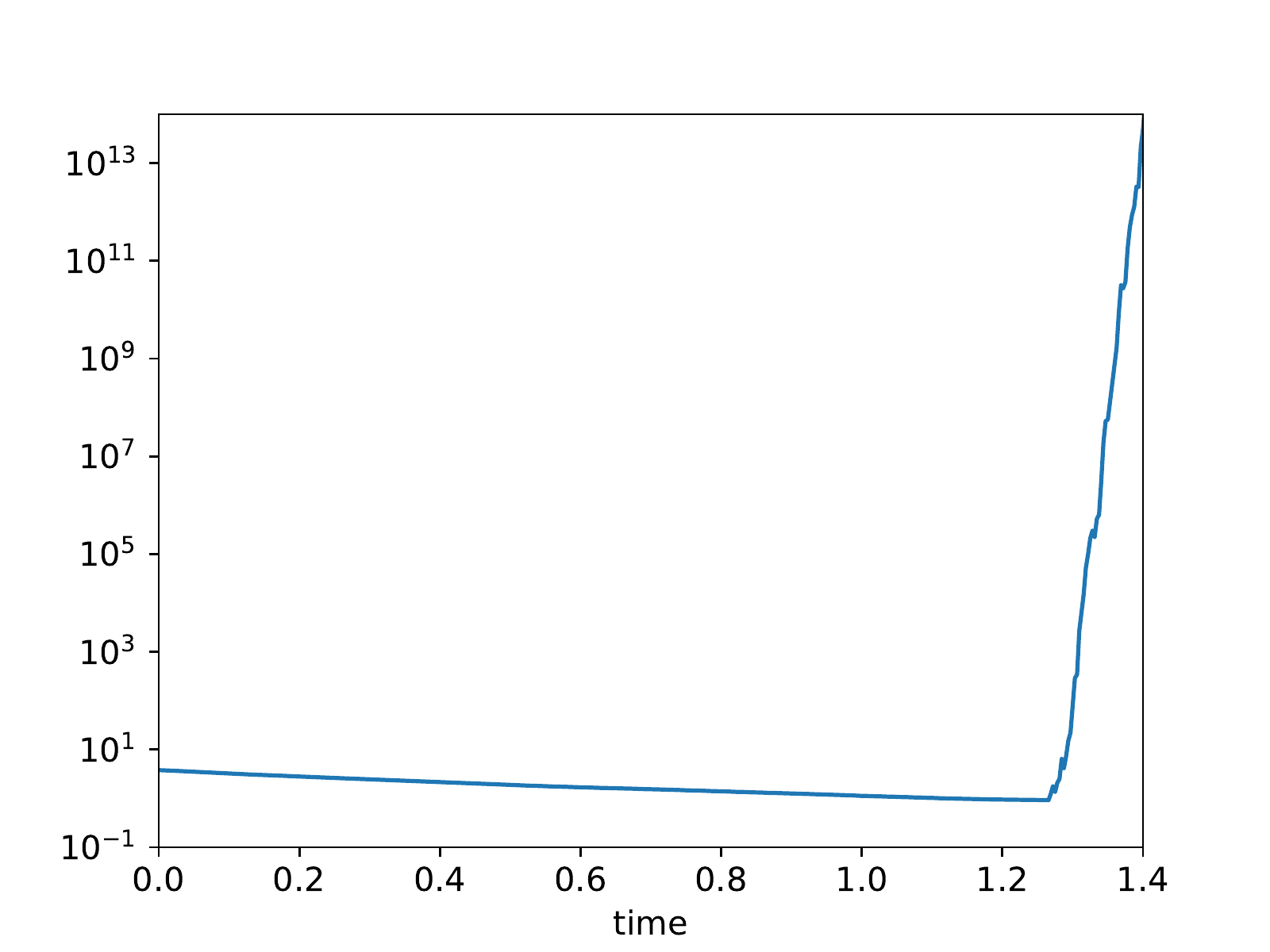}
	    \end{minipage}
	    }
	    \bigskip
	    \subfigure[number of grid points with close eigenvalues, $N=7$]{
	    \begin{minipage}[b]{0.46\textwidth}
	    \includegraphics[width=1\textwidth]{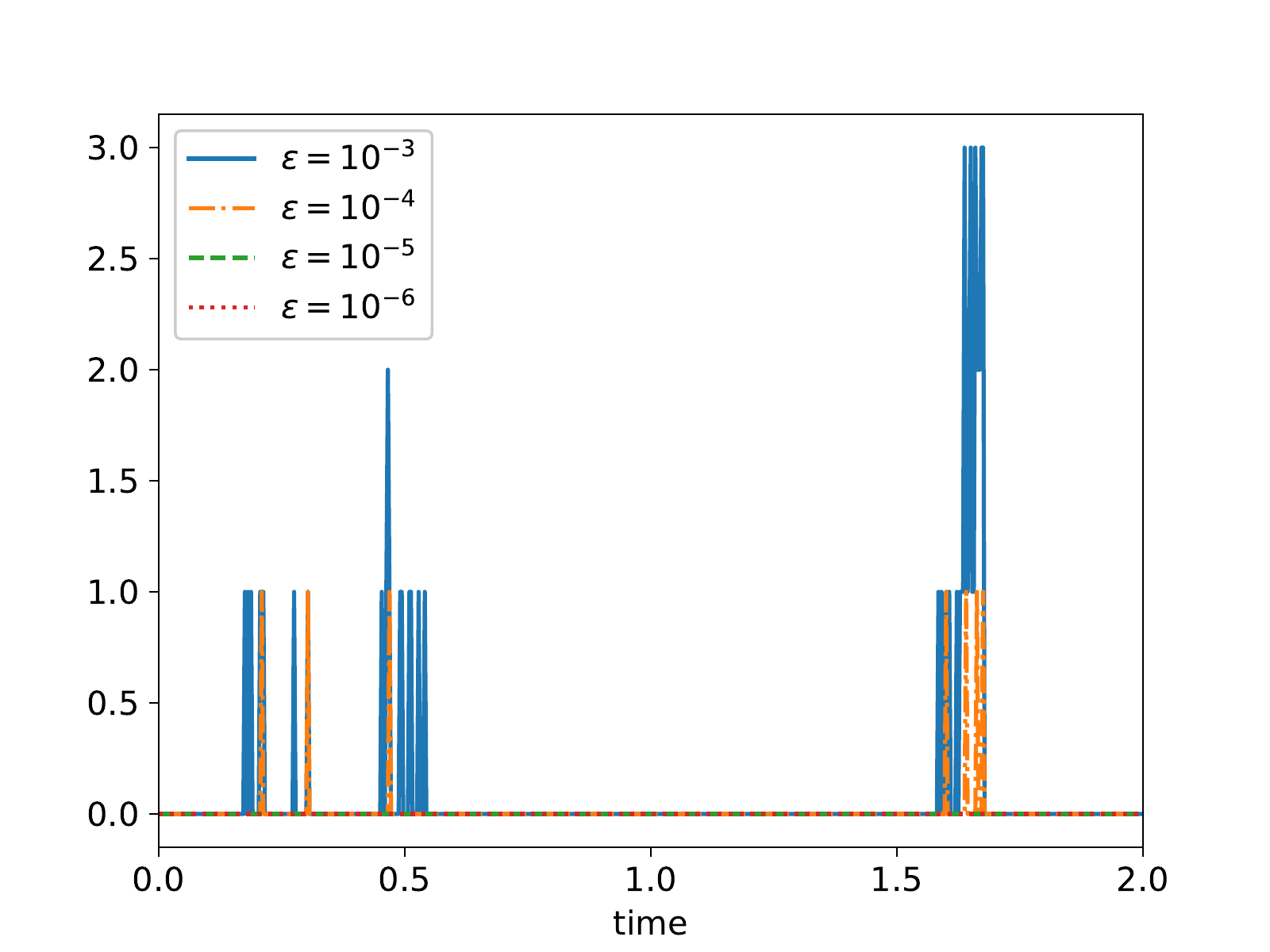}
	    \end{minipage}
	    }
	    \subfigure[$L^{\infty}$ norm of numerical solution, $N=7$]{
	    \begin{minipage}[b]{0.46\textwidth}    
	    \includegraphics[width=1\textwidth]{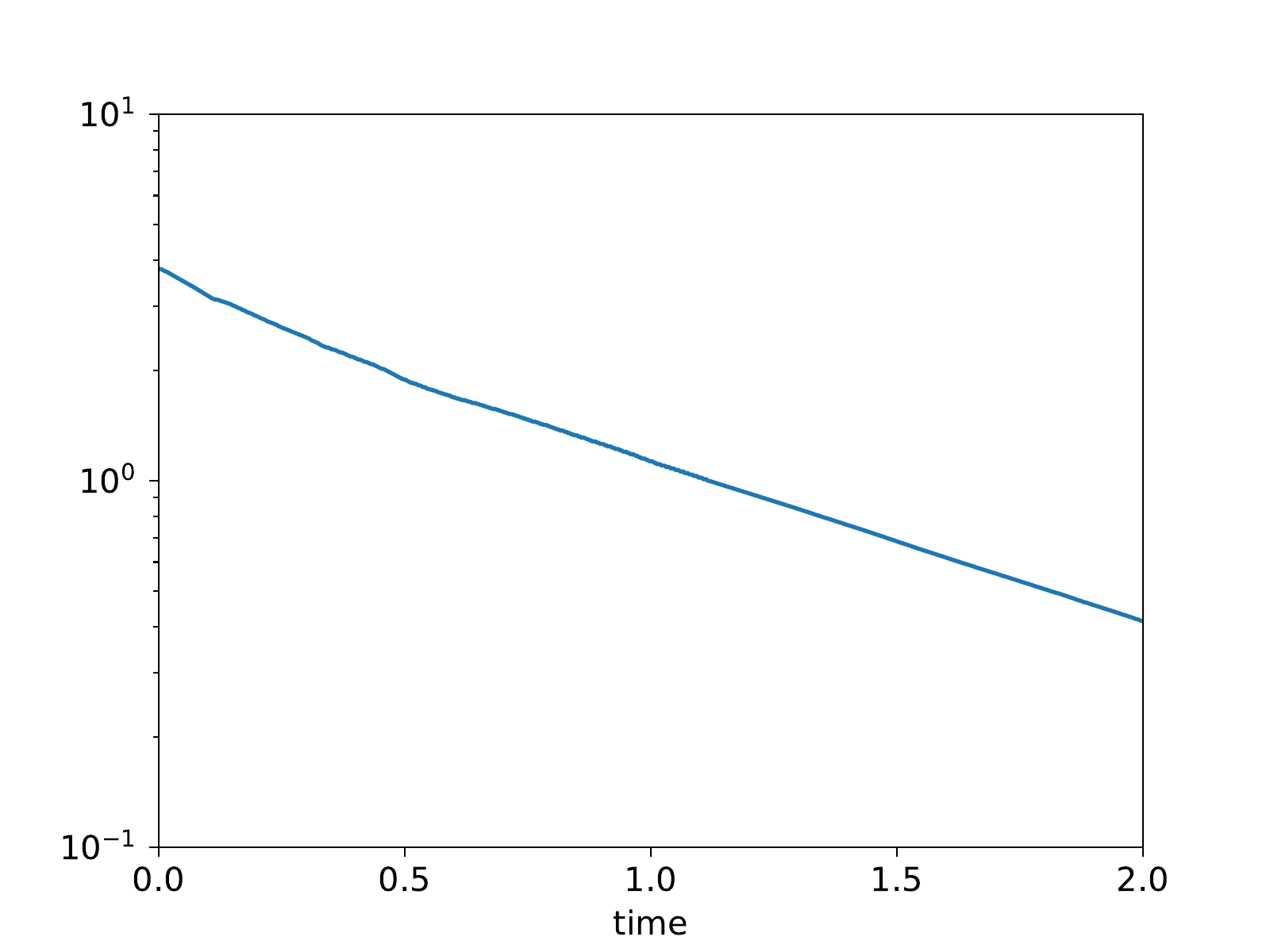}
	    \end{minipage}
	    }	    
		\caption{Example \ref{ex:const}: constant scattering and absorption coefficients.  Here, we use the first neural network architecture in Figure \ref{fig:schematic-nn}. The number of grid points with imaginary eigenvalues and $L^{\infty}$ norm of numerical solutions during the time evolution in the optically thin regime ($\sigma_s=\sigma_a=1$) with $N=3,5,7$.}
	    \label{fig:const-distinct-instability}
	\end{figure}	

We also observe numerical instability in the hyperbolic ML closure model with distinct eigenvalues for some parameters. The model is numerically stable for $N\ge 6$ but numerically unstable for $N=3,4,$ and $5$  in the optically thin regime. We show the distributions of the training data and the numerical solution during the time evolution of the ML closure model with $N=3$ and $N=6$ in Figure \ref{fig:const-training-data-instability}. At each time step, there is a curve composed of 256 points and the plots represent the evolution of the closed curve where the color denotes the evolution time. It can be seen for the $N=3$ case, that as the  numerical solutions is approaching the steady state, it suddenly undergoes a dramatic change in the dynamics of the solution and then proceeds to run out side of the range of the training data.  This in contrast to the case $N=6$ which is plotted in Figure \ref{fig:const-training-data-instability}(b), which clearly shows relaxation to the steady state.  In the plots, the color bar represents the time of the solution.
	\begin{figure}
	    \centering
	    \subfigure[$m_2/m_0$ vs $m_1/m_0$ with $N=3$]{
	    \begin{minipage}[b]{0.48\textwidth}
	    \includegraphics[width=1\textwidth]{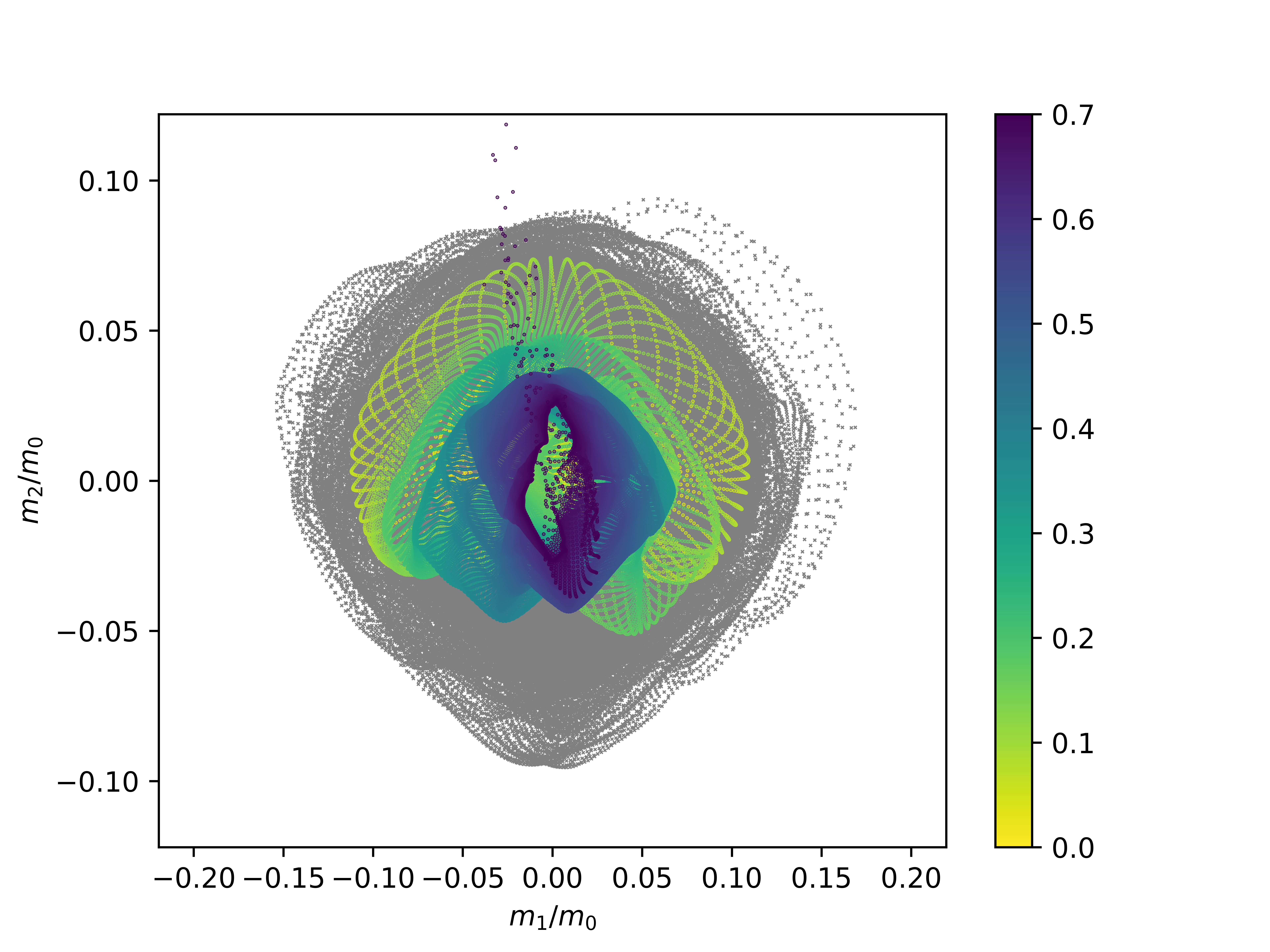}
	    \end{minipage}
	    }
	    \subfigure[$m_2/m_0$ vs $m_1/m_0$ with $N=6$]{
	    \begin{minipage}[b]{0.48\textwidth}    
	    \includegraphics[width=1\textwidth]{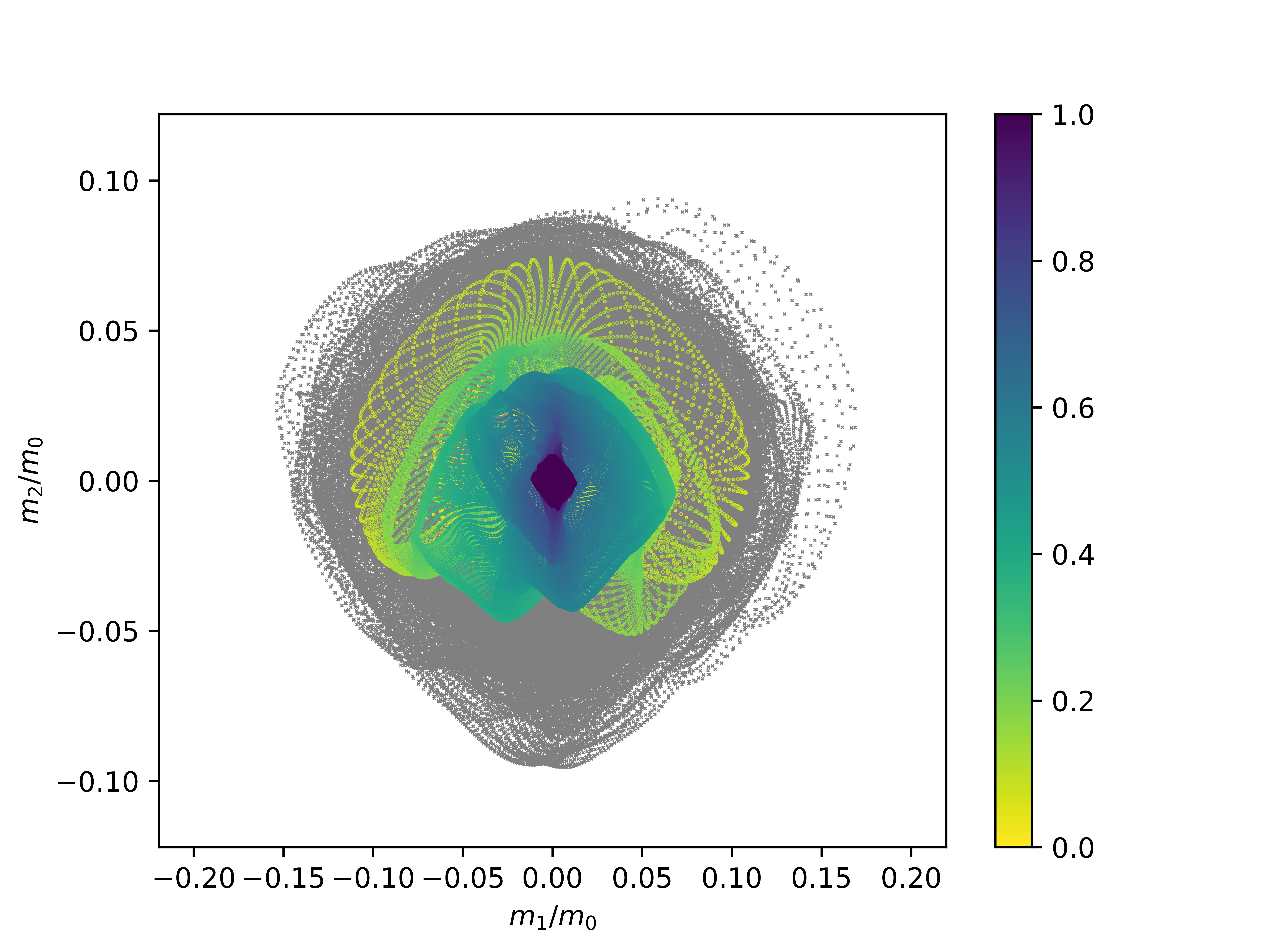}
	    \end{minipage}
	    }	    
		\caption{Example \ref{ex:const}: constant scattering and absorption coefficients, $m_2/m_0$ vs $m_1/m_0$ with $N=3$ and $N=6$.  Here, we use the second neural network architecture in Figure \ref{fig:schematic-nn-distinct}. At each time step, there is a curve composed of 256 points and the plots represent the evolution of the closed curve where the color denotes the evolution time. For $N=3$, as the numerical solutions is approaching the steady state, it suddenly has a dramatic change in dynamics of the solution and then proceeds to run outside of the range of the training data. This in contrast to the case $N=6$ which is plotted, which clearly shows relaxation to the steady state. The grey points denote the training data and the colorful points denote the numerical solutions solving from the ML moment closure system.}
	    \label{fig:const-training-data-instability}
	\end{figure}

To investigate the instability of the ML closure model with distinct eigenvalues further, we check the linear stability of the system numerically. We denote the source term of the closure model in \eqref{eq:closure-model} by $S=( -\sigma_a m_0, -(\sigma_s+\sigma_a) m_{1}, \cdots, -(\sigma_s+\sigma_a) m_{N})$. Then, the Jacobian matrix of the source term is $S_U=\textrm{diag}(-\sigma_a, -(\sigma_s+\sigma_a), \cdots, -(\sigma_s+\sigma_a))$. The model is called linearly stable if all the eigenvalues of $(i\xi A + S_U)$ have non-positive real part for any $\xi\in\mathbb{R}$. Here, $A$ is the coefficient matrix of the closure system given in \eqref{eq:coefficient-matrix} and $i$ is the imaginary unit. Linear stability is essential for the closure system to generate stable results in long time simulations \cite{yong2001basic}. The symmetrizer based hyperbolic ML moment closure model in \cite{huang2021hyperbolic}  satisfies this stability condition. 
We test for linear stability numerically, by taking $\xi = -100,-99,\cdots, 99, 100$, and computing the eigenvalues of $(i\xi A + S_U)$ at all grid points. The number of grid points with eigenvalues with positive real part and the $L^{\infty}$-norm of $m_0$ during the time evolution is shown in Figure \ref{fig:linear-instability}. It is observed that for $N=3$ and 5, the solution blows up when the grid points with linear instability appear. This indicates that the loss of linear stability probably results in the blow up of the numerical solutions. It is also interesting to see that for $N=9$, the model generates stable solution; however, there also exists several grid points with linear instability when the time is around 0.7 and the model returns to stability in the time afterwards. How to stabilize the closure system, while maintaining training accuracy, is a topic to be investigated in the future.
	\begin{figure}
	    \centering
	    \subfigure[number of grid points with linear instability in the time evolution]{
	    \begin{minipage}[b]{0.46\textwidth}
	    \includegraphics[width=1\textwidth]{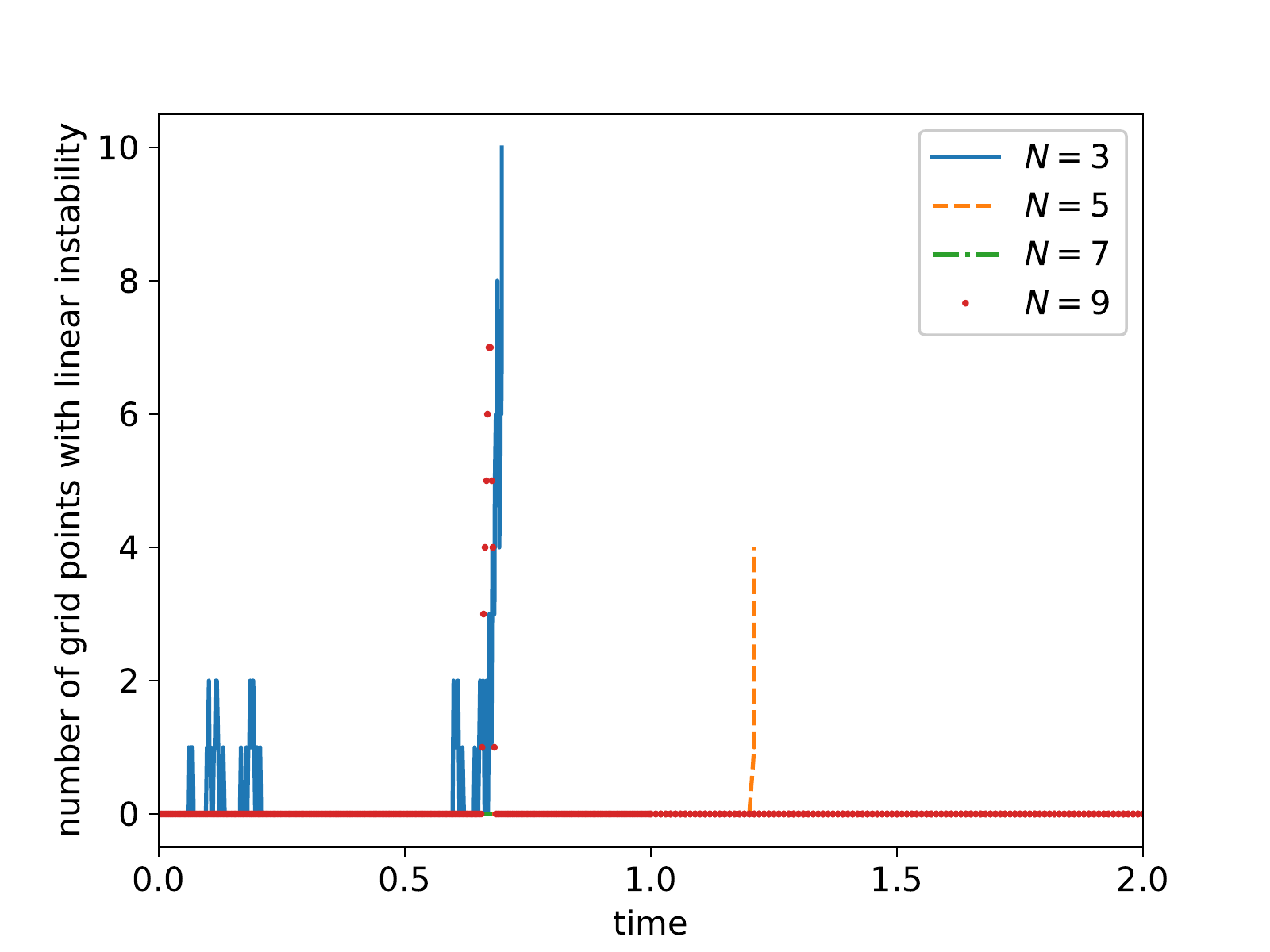}
	    \end{minipage}
	    }
	    \subfigure[$L^{\infty}$-norm of $m_0$ in the time evolution]{
	    \begin{minipage}[b]{0.46\textwidth}    
	    \includegraphics[width=1\textwidth]{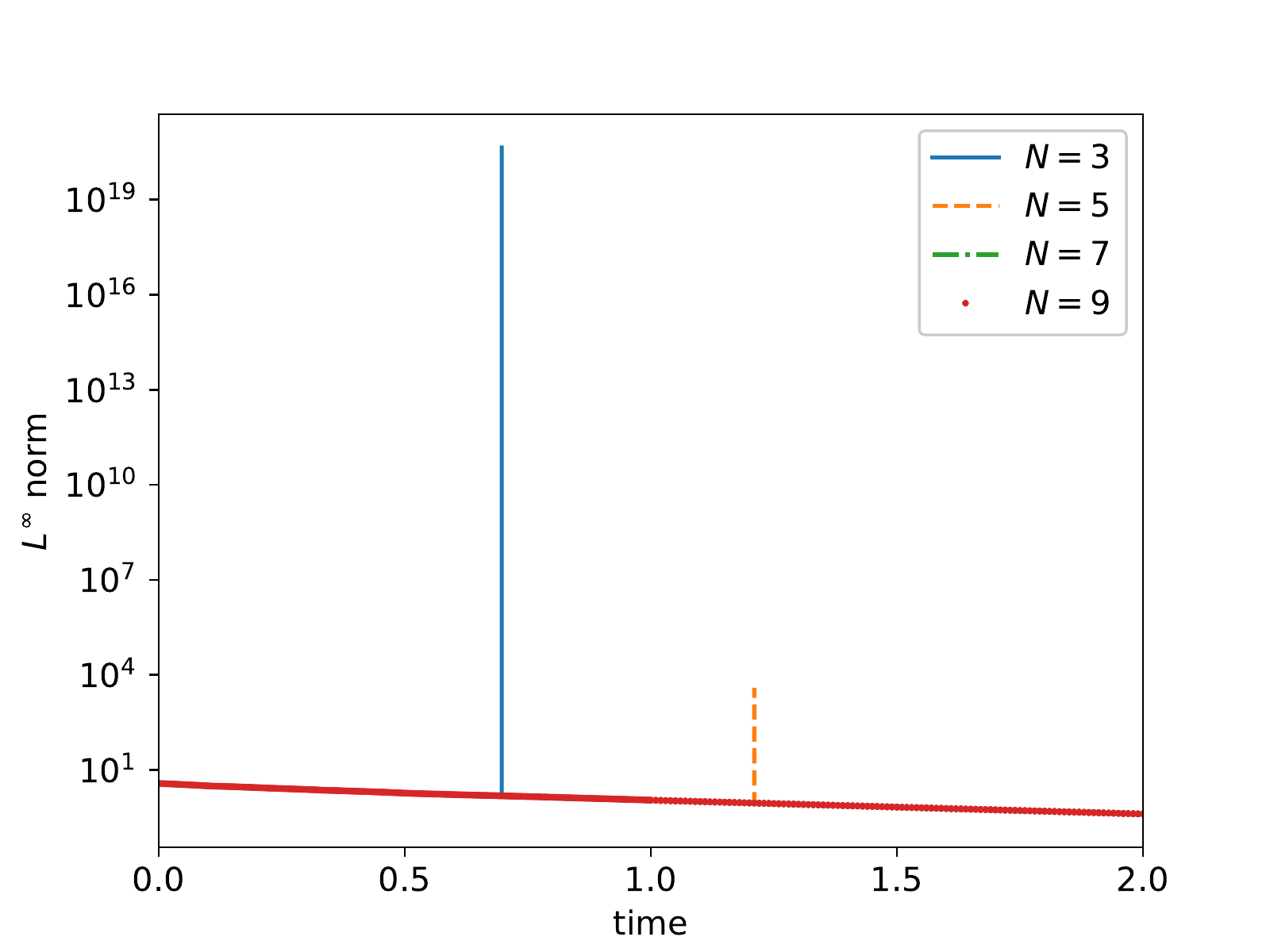}
	    \end{minipage}
	    }	    
		\caption{Example \ref{ex:const}: constant scattering and absorption coefficients.}
	    \label{fig:linear-instability}
	\end{figure}

\end{exam}

\begin{exam}[Gaussian source problem]\label{ex:gauss-source}
	In this example, we investigate the RTE with the initial condition to be a Gaussian distribution in the physical domain:
	\begin{equation}\label{eq:gauss-source-init}
		f_0(x,v) = \frac{c_1}{(2 \pi \theta)^{1/2}} \exp\brac{-\frac{(x - x_0)^2}{2 \theta}} + c_2.
	\end{equation}
	In this test, we take $c_1=0.5$, $c_2=2.5$, $x_0=0.5$ and $\theta=0.01$. 
	We note that this problem is named the Gaussian source problem in the literature \cite{frank2012perturbed,fan2020nonlinear}.

	In Figure \ref{fig:gauss-source-compare}, we present the results obtained using various closure models. Here, we take $\sigma_s=1$ and $\sigma_a=0$. We observe good agreement between the three ML closure models and the kinetic model, while the $P_N$ model has large deviations from the kinetic model. These results show the good generalizability of our ML closure models. Moreover, the three hyperbolic ML models have the same level of accuracy in this test.
	\begin{figure}
	    \centering
	    \subfigure[$m_0$ at $t = 1$]{
	    \begin{minipage}[b]{0.46\textwidth}
	    \includegraphics[width=1\textwidth]{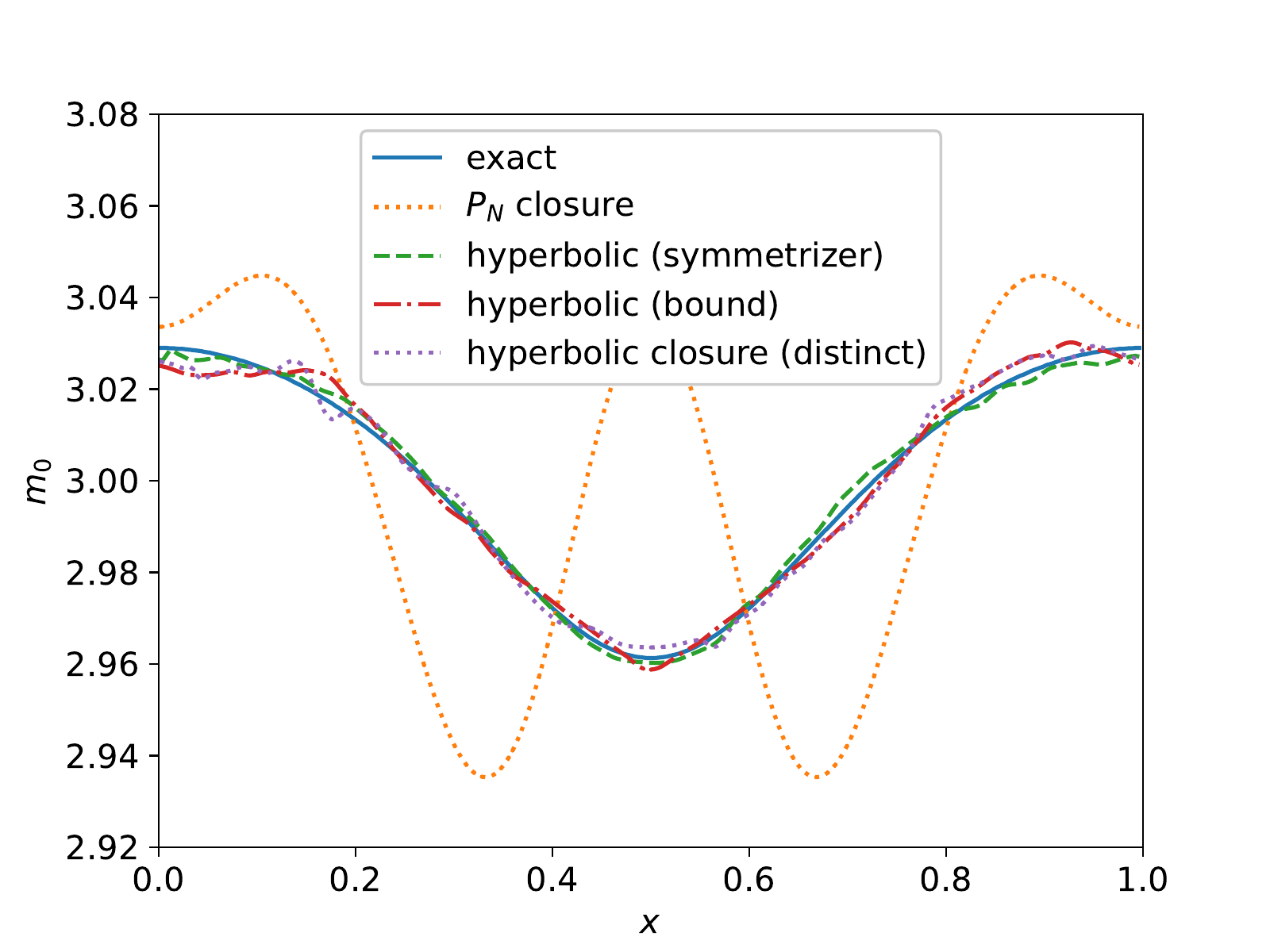}
	    \end{minipage}
	    }
	    \subfigure[$m_1$ at $t = 1$]{
	    \begin{minipage}[b]{0.46\textwidth}    
	    \includegraphics[width=1\textwidth]{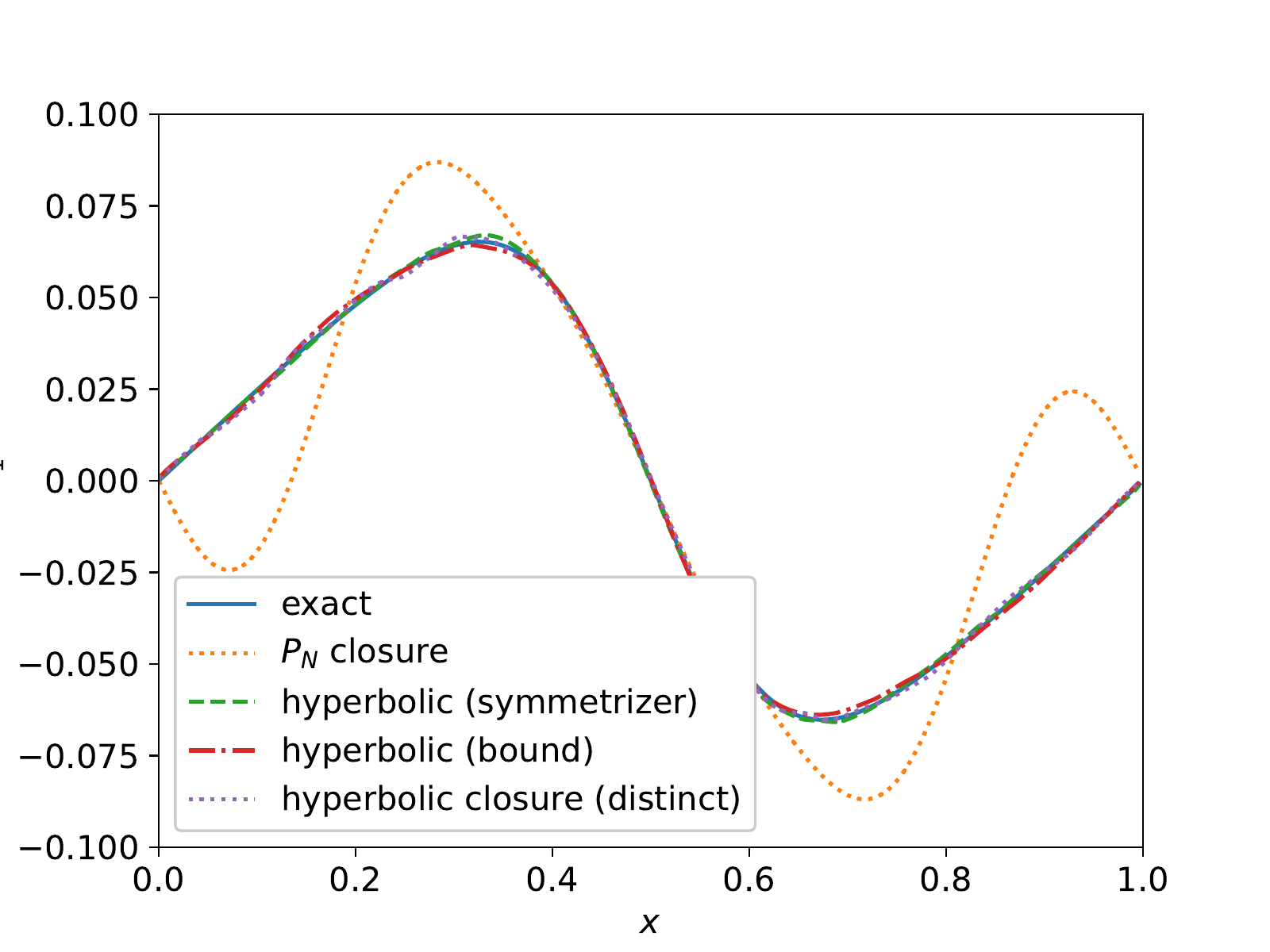}
	    \end{minipage}
	    }	    
		\caption{Example \ref{ex:gauss-source}: Gaussian source problem, $N=6$ and $t=1$.}
	    \label{fig:gauss-source-compare}
	\end{figure}
\end{exam}

\begin{exam}[two-material problem]\label{ex:two-material}
	The two-material problem models a domain with a discontinuous material cross section \cite{larsen1989asymptotic}. In our problem setup, there exist two discontinuities $0<x_1<x_2<1$ in the domain, and $\sigma_s$ and $\sigma_a$ are piecewise constant functions:
	$$ 
	\sigma_s(x)=\left\{
	\begin{aligned}
	& \sigma_{s1}, \quad  ~ x_1 < x < x_2, \\
	& \sigma_{s2}, \quad  ~ 0\le x < x_1 ~ \textrm{or} ~ x_2\le x < 1.
	\end{aligned}
	\right.
	$$
	and
	$$ 
	\sigma_a(x)=\left\{
	\begin{aligned}
	& \sigma_{a1}, \quad  ~ x_1 < x < x_2, \\
	& \sigma_{a2}, \quad  ~ 0\le x < x_1 ~ \textrm{or} ~ x_2\le x < 1.
	\end{aligned}
	\right.
	$$

	Specifically, we take $x_1=0.3$, $x_2=0.7$, $\sigma_{s1}=1$, $\sigma_{s2}=10$ and $\sigma_{a1}=\sigma_{a2}=0$. The numerical results are shown in Figure \ref{fig:two-material-N6}. The gray part is in the optically thin regime and the other part is in the intermediate regime. We observe that our current closure model agrees well with the kinetic solution over the whole domain at both $t=0.5$ and $t=1$. We note that this is in contrast to the $P_N$ closure, which has large deviations from the kinetic solution in the optically thin portion of the domain, see Figure \ref{fig:two-material-N6}. Moreover, the two eigenvalue based hyperbolic closures perform better than the closure in \cite{huang2021hyperbolic} which has some overshoot near the discontinuities, see Figure \ref{fig:two-material-N6} (d).
	\begin{figure}
	    \centering
	    \subfigure[$m_0$ at $t=0.5$]{
	    \begin{minipage}[b]{0.46\textwidth}
	    \includegraphics[width=1\textwidth]{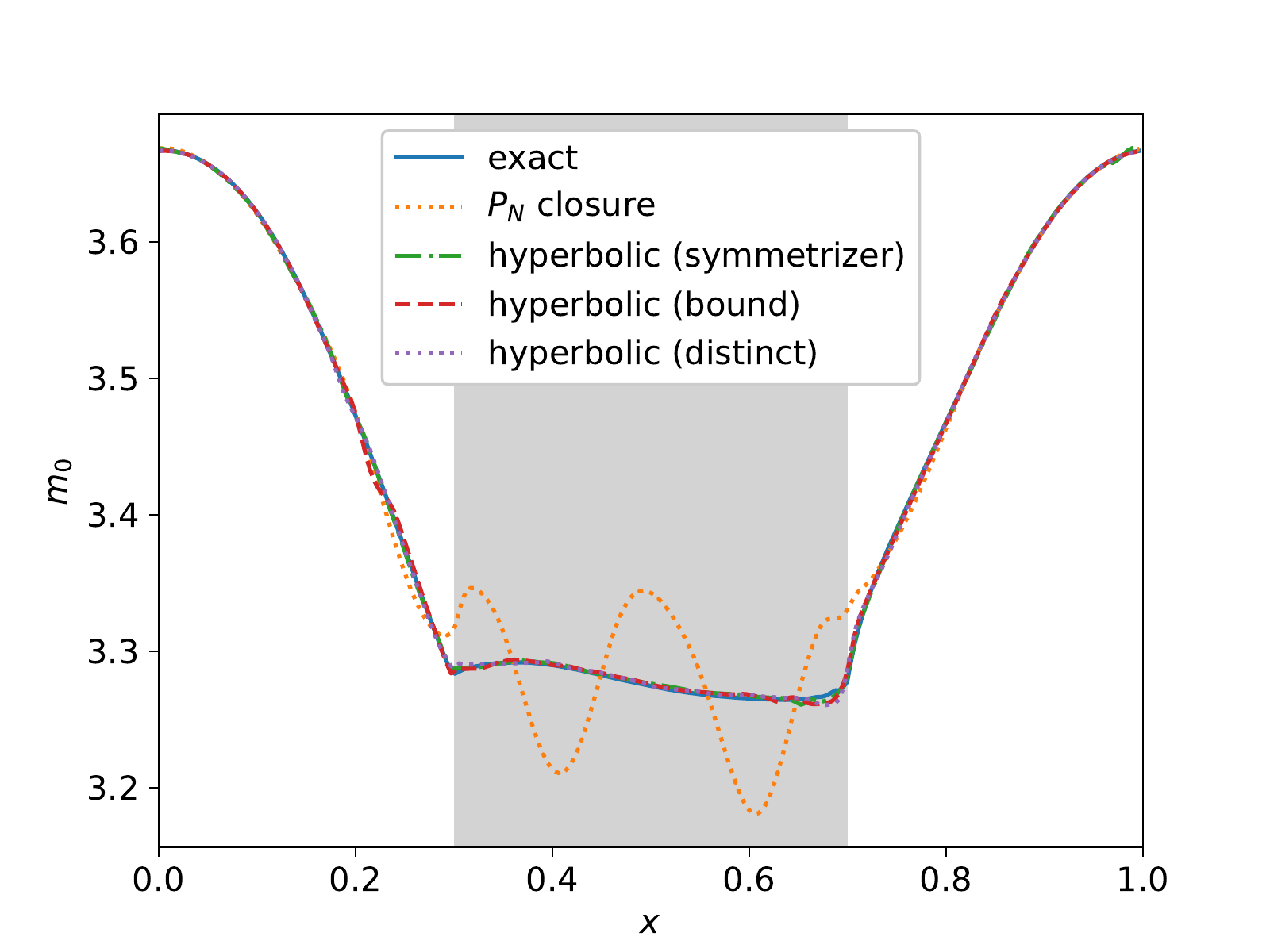}
	    \end{minipage}
	    }
	    \subfigure[$m_1$ at $t=0.5$]{
	    \begin{minipage}[b]{0.46\textwidth}    
	    \includegraphics[width=1\textwidth]{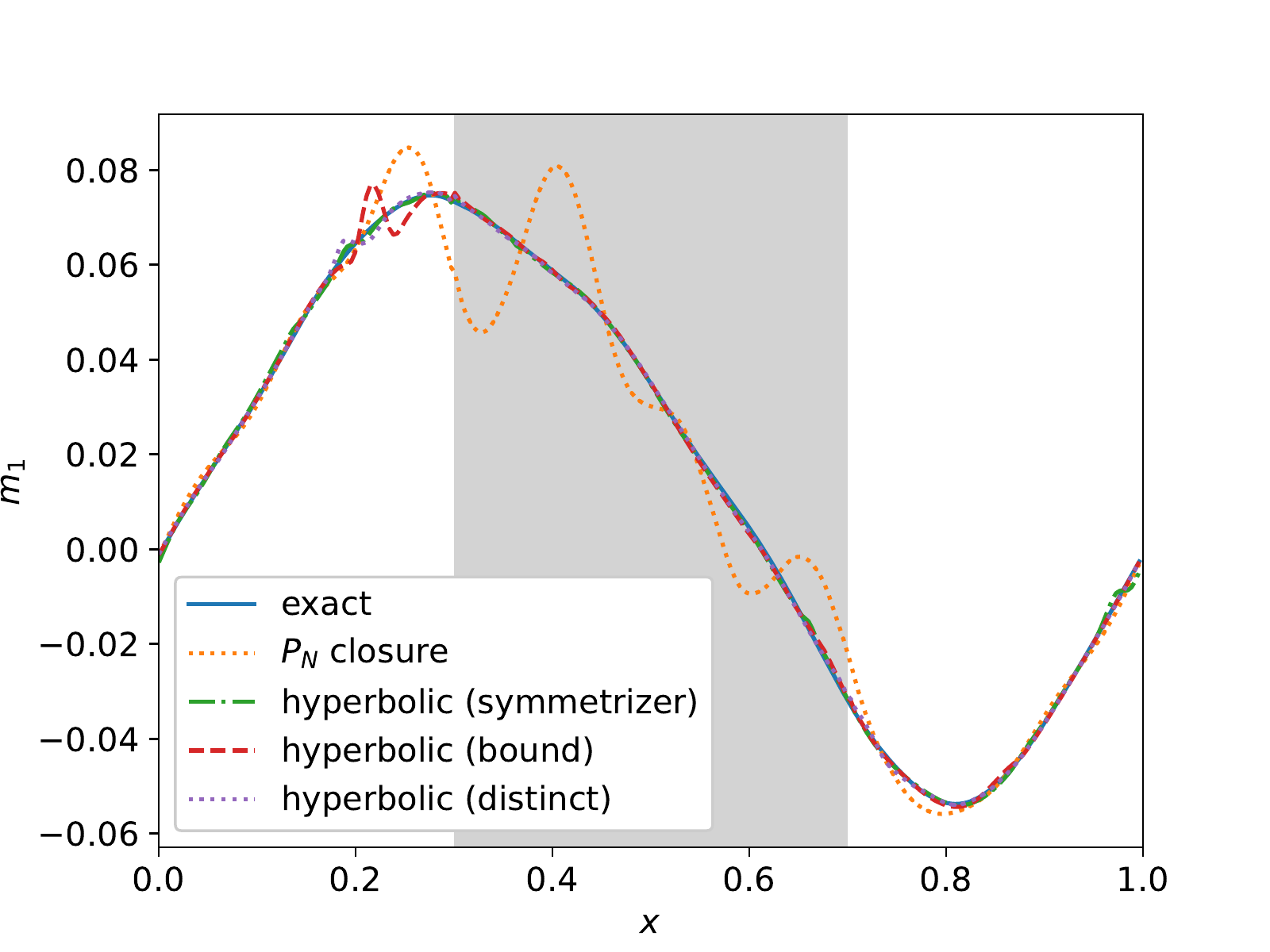}
	    \end{minipage}
	    }
	    \bigskip
	    \subfigure[$m_0$ at $t=1$]{
	    \begin{minipage}[b]{0.46\textwidth}
	    \includegraphics[width=1\textwidth]{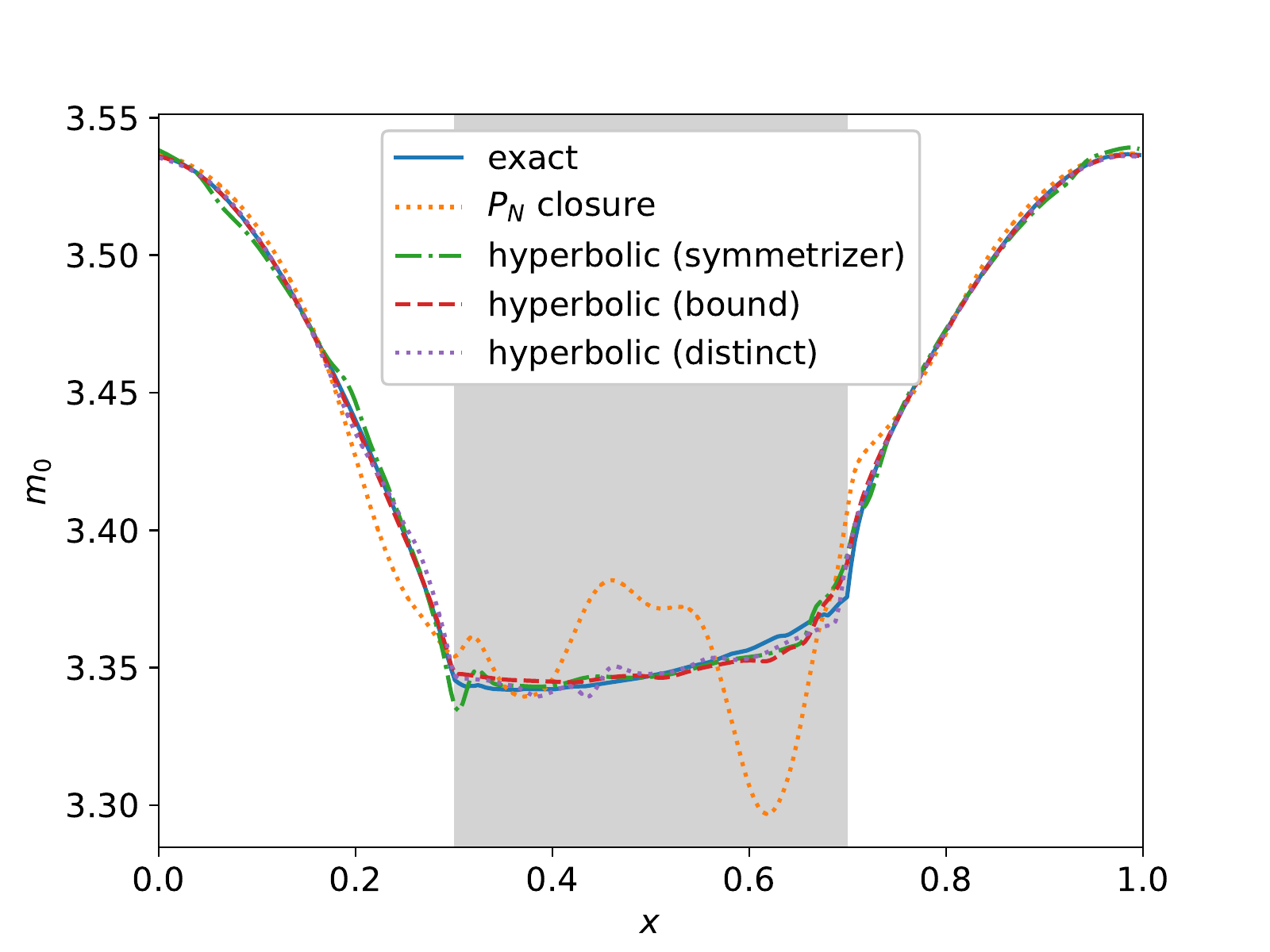}
	    \end{minipage}
	    }
	    \subfigure[$m_1$ at $t=1$]{
	    \begin{minipage}[b]{0.46\textwidth}    
	    \includegraphics[width=1\textwidth]{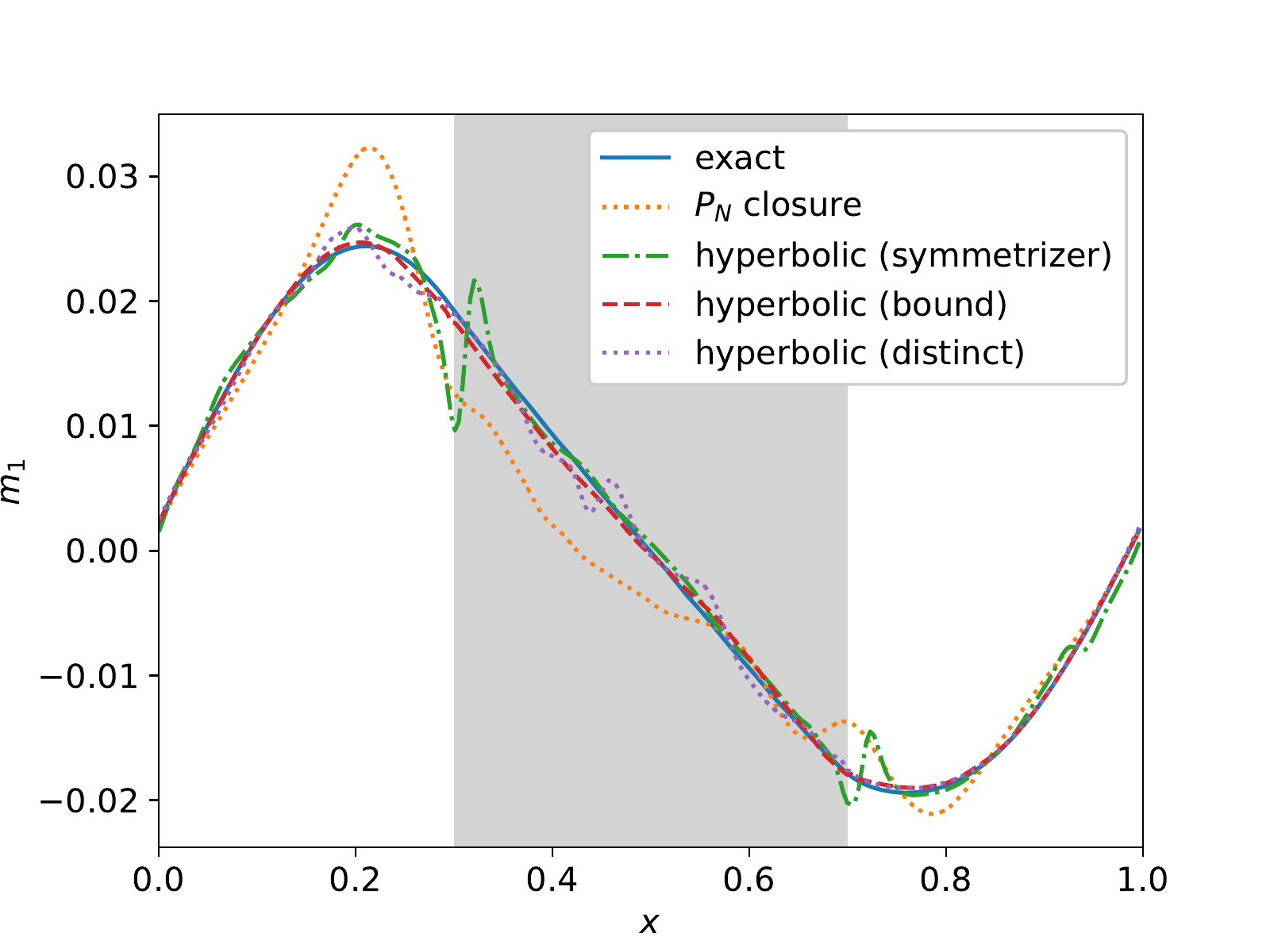}
	    \end{minipage}
	    }
		\caption{Example \ref{ex:two-material}: two-material problem. Numerical solutions of $m_0$ and $m_1$ at $t=0.5$ and $t=1$ with $N=6$. The gray part in the middle is in the optically thin regime and the other part is in the intermediate regime.}
	    \label{fig:two-material-N6}
	\end{figure}

	In Figure \ref{fig:two-material-convergence}, we numerically investigate the convergence of the ML closure model with bounded eigenvalues to the kinetic model as the number of moments increases. We take the number of moment to be $N=6,8,10,12,14,16$. In Table \ref{tab:table-two-material}, we present the relative $L^2$ errors of $m_0$ and $m_1$ for the same numerical example. We observe that the error between the solution to the ML closure model and the solution to the kinetic equation becomes smaller with an increasing number of moments. This numerically indicates that the ML closure model converges to the kinetic model as the number of moments increases.  It is worth noting that the saturation in convergence seen in table 5.1 is of the same order as the training error in the ML Closure model.
	\begin{figure}
	    \centering
	    \subfigure[$m_0$ at $t=2$]{
	    \begin{minipage}[b]{0.46\textwidth}
	    \includegraphics[width=1\textwidth]{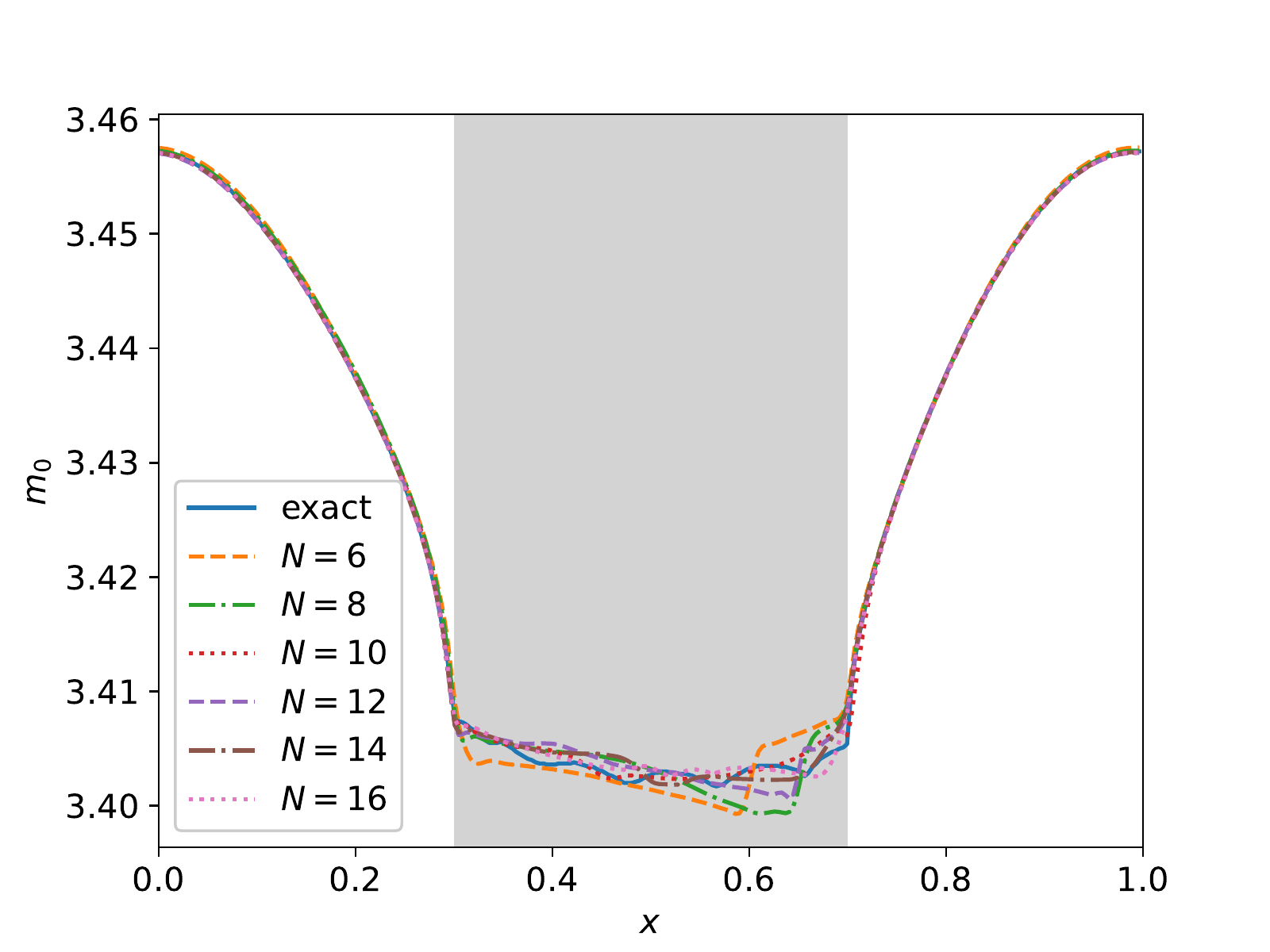}
	    \end{minipage}
	    }
	    \subfigure[$m_1$ at $t=2$]{
	    \begin{minipage}[b]{0.46\textwidth}    
	    \includegraphics[width=1\textwidth]{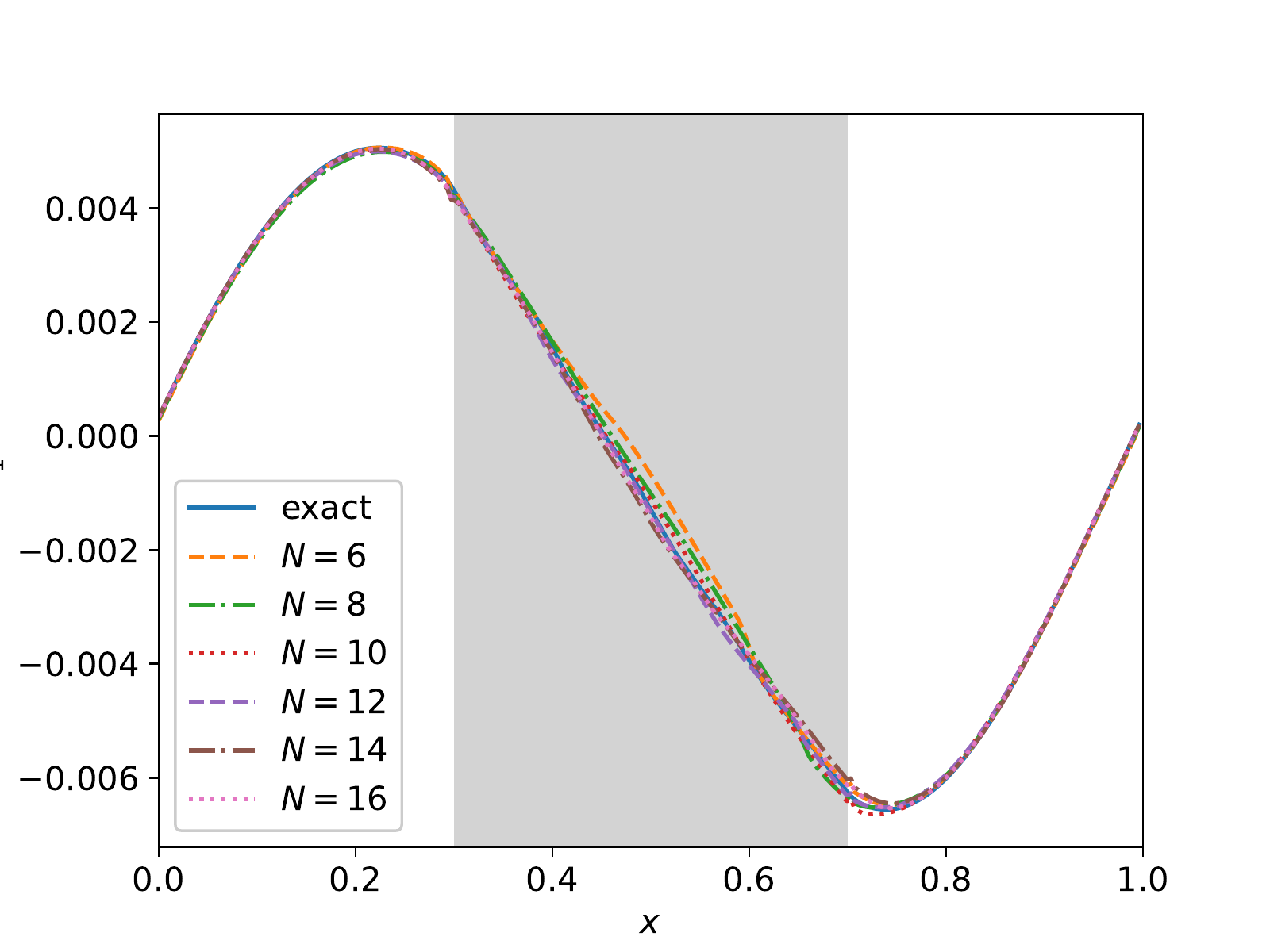}
	    \end{minipage}
	    }
		\caption{Example \ref{ex:two-material}: two-material problem, convergence with respect to number of moments, the ML closure model with bounded eigenvalues. Numerical solutions of $m_0$ at $t=2$ with $N=6,8,10,12,14,16$. The gray part in the middle is in the optically thin regime and the other part is in the intermediate regime. }
	    \label{fig:two-material-convergence}
	\end{figure}

    \begin{table}[!hbp]
    \centering
    \caption{Example \ref{ex:two-material}: two-material problem, convergence with respect to number of moments, the ML closure model with bounded eigenvalues. The relative $L^2$ errors of the numerical solutions of $m_0$ and $m_1$ at $t=2$ with $N=6,8,10,12,14,16$.}
    \label{tab:table-two-material}
    \begin{tabular}{c|c|c}
      \hline
    $N$  & relative $L^2$ error of $m_0$ & relative $L^2$ error of $m_1$ \\
      \hline
    6   & 5.79e-4   & 7.84e-2 \\ 
    8   & 3.78e-4   & 5.69e-2 \\
    10  & 3.67e-4   & 2.93e-2 \\
    12  & 3.66e-4   & 3.64e-2   \\
    14  & 1.98e-4   & 2.94e-2   \\
    16  & 1.66e-4   & 2.21e-2   \\
    \hline
    \end{tabular}
    \end{table}

\end{exam}

\section{Conclusion}\label{sec:conclusion}

In this paper, we propose a new method to enforce the hyperbolicity of a ML closure model. Motivated by the observation that the coefficient matrix of the closure system is a lower Hessenberg matrix, we relate its eigenvalues to the roots of an associated polynomial. We design two new neural network architectures based on this relation.
The ML closure model resulting from the first neural network is weakly hyperbolic and guarantees the physical characteristic speeds, i.e. the eigenvalues lie in the range of the interval $[-1, 1]$.
The second model is strictly hyperbolic, but does not guarantee the boundedness of the eigenvalues, although in practice the eigenvalues lie nearly within the physical range.
Having the physical characteristic speeds saves substantial computational expenses when numerically solving the closure system by allowing for a larger time step size compared to \cite{huang2021hyperbolic}. 
Several benchmark tests including the Gaussian source problem and the two-material problem show the good accuracy and generalizability of our hyperbolic ML closure model. Nevertheless, there exists some numerical instability for the current model when a small number of moments are used. We will try to fix this problem in the future work.

\section*{Acknowledgment}

We thank Michael M. Crockatt in Sandia National Laboratories for providing numerical solver for the radiative transfer equation. We acknowledge the High Performance Computing Center (HPCC) at Michigan State University for providing computational resources that have contributed to the research results reported within this paper. JH would like to thank Professor Wen-An Yong in Tsinghua University for many fruitful discussions. This work has been assigned a document release number LA-UR-21-28626.

\begin{appendices}

\section{Collections of proofs}\label{appendix:proof}

In this appendix, we collect some lemma and proofs. We start with a lemma which characterize the eigenspace of unreduced lower Hessenberg matrix.
\begin{lem}\label{lem:geometric-multiplicity}
	For an unreduced lower Hessenberg matrix $H=(h_{ij})_{n\times n}$, the geometric multiplicity of any eigenvalue $\lambda$ is 1 and the corresponding eigenvector is $(q_0(\lambda), q_1(\lambda), \cdots, q_{n-1}(\lambda))^T$. Here $\{q_i\}_{0\le i\le n-1}$ is the associated polynomial sequence defined in \eqref{eq:recurrence-polynomial}.
\end{lem}
\begin{proof}	
	By Definition \ref{defn:hessenberg}, we have that $h_{ij}=0$ for $j>i+1$ and $h_{i,i+1}\ne0$ for $i=1,\cdots,n-1$.
	Let $r=(r_1,r_2,\cdots,r_n)$ be an eigenvector associated with $\lambda$. We write $Ar = \lambda r$ as an equivalent component-wise formulation:
	\begin{equation}\label{eq:geo-eq-1}
		\sum_{j=1}^i h_{ij} r_j + h_{i,i+1} r_{i+1} = \lambda r_i, \quad i = 1,\cdots,n-1,
	\end{equation}
	and
	\begin{equation}\label{eq:geo-eq-2}
		\sum_{j=1}^n h_{nj} r_j = \lambda r_n.
	\end{equation}
	Here we used the fact that $h_{ij}=0$ for $j>i+1$.
	Since $h_{i,i+1}\ne0$ for $i=1,\cdots,n-1$, \eqref{eq:geo-eq-1} is equivalent to
	\begin{equation}\label{eq:geo-eq-1-equiv}
		r_{i+1} = \frac{1}{h_{i,i+1}} \brac{\lambda r_i - \sum_{j=1}^i h_{ij} r_j}, \quad i = 1,\cdots,n-1
	\end{equation}
	From \eqref{eq:geo-eq-1-equiv}, we deduce that $r_1\ne0$, otherwise $r_2=\cdots=r_n=0$. Moreover, $r_i$ for $i = 2,\cdots,n$ are uniquely determined by $r_1$. Therefore, the geometric multiplicity of $\lambda$ is 1. Moreover, without loss of generality, we take $r_1=1$. In this case, $r$ is exactly the same with $(q_0(\lambda),q_1(\lambda),\cdots,q_{n-1}(\lambda))^T$. Here $\{q_i\}_{0\le i\le n-1}$ is the associated polynomial sequence defined in \eqref{eq:recurrence-polynomial}.
\end{proof}

\begin{lem}\label{lem:eigenvalue-is-root}
	Let $H = (h_{ij})_{n\times n}$ be an unreduced lower Hessenberg matrix and $\{q_i\}_{0\le i\le n}$ is the associated polynomial sequence with $H$. If $\lambda$ is an eigenvalue of $H$, then $\lambda$ is a root of $q_n$.
\end{lem}
\begin{proof}
	From Lemma \ref{lem:geometric-multiplicity}, we have the geometric multiplicity of $\lambda$ is 1 and the corresponding eigenvector $\vect{q}_{n-1}(\lambda) = (q_0(\lambda),q_1(\lambda),\cdots,q_{n-1}(\lambda))^T$, i.e. $H \vect{q}_{n-1}(\lambda) = \lambda \vect{q}_{n-1}(\lambda)$. Plugging $\lambda$ into \eqref{eq:recurrence-polynomial-matx-vect}, we immediately have $q_n(\lambda)=0$, i.e., $\lambda$ is a root of $q_n$.
\end{proof}

\subsection{Proof of Theorem \ref{thm:real-diagonalizable}}
\begin{proof}
	We start by proving that condition 1 and condition 2 are equivalent. First, it is easy to see that condition 2 implies condition 1. We only need to prove that condition 1 implies condition 2. Since $A$ is real diagonalizable, all the eigenvalues of $A$ are real. Moreover, for any eigenvalue of $A$, the geometric multiplicity is equal to its algebraic multiplicity. By Lemma \ref{lem:geometric-multiplicity}, the geometric multiplicity of any eigenvalue of an unreduced lower Hessenberg matrix is 1. Therefore, any eigenvalue of $A$ has algebraic multiplicity of 1, i.e. all the eigenvalues of $A$ are distinct.

	Next, we prove that the equivalence of condition 2 and condition 3. It is easy to see that, condition 3 implies condition 2 from Theorem \ref{thm:root-is-eigenvalue}, and condition 2 implies condition 3 from Lemma \ref{lem:eigenvalue-is-root}. This completes the proof.
\end{proof}

\subsection{Proof of Lemma \ref{lem:monomial-to-legendre}}

\begin{proof}

We start from the definition of Legendre polynomials by the generating function:
\begin{equation}\label{eq:generating-function}
	\frac{1}{\sqrt{1-2tx+t^2}} = \sum_{n=0}^\infty P_n(x) t^n.
\end{equation}
Introduce the variable $s$ such that
\begin{equation}\label{eq:x-to-s}
	1 - ts = \sqrt{1-2tx + t^2},
\end{equation}
which is equivalent to
\begin{equation}\label{eq:s-to-x}
	x = \frac{1+t^2 - (1-ts)^2}{2t} = s + \frac{t}{2}(1-s^2).
\end{equation}
Therefore, we have
\begin{equation}\label{eq:seriez}
	\sum_{n=0}^\infty t^n \int_{-1}^1 x^m P_n(x) dx
	\stackrel{\text{\eqref{eq:generating-function}}}{=} \int_{-1}^1 \frac{x^m dx}{\sqrt{1-2tx+t^2}}
	\stackrel{\text{\eqref{eq:s-to-x}}}{=} \int_{-1}^1 \frac{x^m (1-ts) ds}{\sqrt{1-2tx+t^2}}
	\stackrel{\text{\eqref{eq:x-to-s}-\eqref{eq:s-to-x}}}{=} \int_{-1}^1 \left(s + \frac{t}{2}(1-s^2)\right)^m ds.
\end{equation}

Define 
\begin{equation}
	a_{m,n} := \int_{-1}^1 x^m P_n(x) dx.
\end{equation}
By comparing the coefficients of $t^n$ on both sides of \eqref{eq:seriez}, we find that $a_{m,n} = 0$ if $n>m$ or $m$, $n$ has different parity. For $n = m - 2k$ for some integer $k\ge0$, we have
\begin{equation}
	a_{m,m-2k} = 2^{2k-m}\binom{m}{2k}\int_{-1}^1 s^{2k} (1-s^2)^{m-2k} ds = 2^{2k-m}\binom{m}{2k}\int_{0}^1 2 s^{2k} (1-s^2)^{m-2k} ds
\end{equation}
By introducing the variable $\tau = s^2$ or equivalently $s = \tau^{\frac{1}{2}}$, we have
\begin{equation}
\begin{aligned}
	a_{m,m-2k} &= 2^{2k-m}\binom{m}{2k}\int_{0}^1 2 s^{2k} (1-s^2)^{m-2k} ds \\
	&= 2^{2k-m}\binom{m}{2k}\int_{0}^1 2 \tau^{k} (1-\tau)^{m-2k} \frac{1}{2}\tau^{-\frac{1}{2}} d\tau \\
	&= 2^{2k-m}\binom{m}{2k}\int_{0}^1 \tau^{k-\frac{1}{2}} (1-\tau)^{m-2k} d\tau \\
	&= 2^{2k-m}\binom{m}{2k}\frac{\Gamma(k+\frac12)\Gamma(m-2k+1)}{\Gamma(m-k+\frac32)} \\
	&= \frac{m!}{2^{k-1}k!(2m-2k+1)!!}
\end{aligned}	
\end{equation}
where in the fourth equality we used the relation between the gamma function and the beta function:
\begin{equation}
	B(x,y) := \int_0^1 t^{x-1} (1-t)^{y-1} dt = \frac{\Gamma(x)\Gamma(y)}{\Gamma(x+y)},
\end{equation}
and in the last equality we used the properties of the gamma function: for any integer $n\ge0$
\begin{equation}
	\Gamma(n) = (n-1)!, \quad \Gamma(n+\frac{1}{2}) = \frac{(2n-1)!!}{2^n}\sqrt{\pi}.
\end{equation}

Lastly, using the orthogonality relation $\int_{-1}^1 P_m(x)P_n(x) = \frac{2}{2m+1}\delta_{m,n}$, we have for any integer $m\ge0$,
\begin{equation}
	x^m = \sum_{k=0}^{\lfloor m/2\rfloor} \left(\frac{2m-4k+1}{2}\right) a_{m,m-2k} P_{m-2k}(x)
	= \sum_{k=0}^{\lfloor m/2\rfloor}\frac{m!(2m-4k+1)}{2^k k!(2m-2k+1)!!}P_{m-2k}(x)	
\end{equation}
This completes the proof.
\end{proof}

\end{appendices}

\newpage
\bibliographystyle{abbrv}
\bibliography{ref}

\end{document}